\documentclass[12pt,reqno,hidelinks]{amsart}
\usepackage{amssymb}
\usepackage{amsmath, amssymb, verbatim, url, mathrsfs}
\usepackage{mathrsfs}
\usepackage{dutchcal}
\usepackage{mathtools}
\usepackage{verbatim}
\usepackage{amsthm}
\usepackage{framed}
\usepackage{cite}
\usepackage{wasysym}
\usepackage{upgreek}
\usepackage{color}
\usepackage[dvipsnames]{xcolor}
\usepackage{tensor}
\usepackage{accents}
\usepackage{dsfont}
\usepackage[colorlinks,linkcolor=blue,citecolor=blue]{hyperref}
\usepackage{enumerate}
\usepackage[normalem]{ulem}
\usepackage{longtable}
\usepackage{mathtools}

\numberwithin{equation}{section}

\newtheorem{proposition}{Proposition}[section]
\newtheorem{lemma}[proposition]{Lemma}
\newtheorem{corollary}{Corollary}[section]
\newtheorem{theorem}{Theorem}[section]

\theoremstyle{definition}
\newtheorem{definition}{Definition}[section]
\newtheorem{remark}{Remark}[section]

\newcommand{\uu}{\underline{u}}	
\newcommand{\uuo}{\underline{u_0}}	
\newcommand{\uui}{\underline{u_i}}	
\newcommand{\ca}{\mathcal{a}}
\newcommand{\ba}{\bar{\mathcal{a}}}
\newcommand{\cb}{\mathcal{b}}
\newcommand{\cg}{\mathcal{g}}

\newcommand{\cc}{\mathcal{c}}
\newcommand{\cm}{\mathcal{m}}
\newcommand{\bc}{\bar{\mathcal{c}}}

\newcommand{\ck}{\mathcal{k}}
\newcommand{\mf}{\mathring{f}}

\newcommand{\vertiiii}[1]{{\left\vert\kern-0.25ex\left\vert\kern-0.25ex\left\vert\kern-0.25ex\left\vert #1 \right\vert\kern-0.25ex\right\vert\kern-0.25ex\right\vert\kern-0.25ex\right\vert}}
\newcommand{\vertiii}[1]{{\left\vert\kern-0.25ex\left\vert\kern-0.25ex\left\vert #1 \right\vert\kern-0.25ex\right\vert\kern-0.25ex\right\vert}}

\newcommand{\Rbb}{\mathbb{R}}
\newcommand{\Zbb}{\mathbb{Z}}
\newcommand{\Tbb}{\mathbb{T}}
\newcommand{\Pbb}{\mathbb{P}}

\newcommand{\del}[1]{{\partial_{#1}}}

\newcommand{\AND}{{\quad\text{and}\quad}}

\newcommand{\Li}{L^\infty}
\newcommand{\la}{\langle}
\newcommand{\ra}{\rangle}

\newcommand{\p}[1]{
	\begin{pmatrix}
		#1
	\end{pmatrix}
}

\newcommand{\B}{\mathcal{B}}

\newcommand{\U}{\mathcal{U}}

\newcommand{\Pbp}{\mathbb{P}^{\perp}}

\newcommand{\be}{\begin{equation}}
	\newcommand{\ee}{\end{equation}}

\allowdisplaybreaks

\begin{document}

	\title[Blowups for a class of second order nonlinear hyperbolic equations]{Blowups for a class of second order nonlinear hyperbolic equations: A reduced model of nonlinear Jeans instability}

	\author{Chao Liu}
	
	\address[Chao Liu]{Center for Mathematical Sciences and School of Mathematics and Statistics, Huazhong University of Science and Technology, Wuhan 430074, Hubei Province, China.}
	\email{chao.liu.math@foxmail.com}

\begin{abstract}
Understanding the formation of nonlinear structures in the universe and stellar systems is crucial. The nonlinear Jeans instability plays a key role in these formation processes. It has been a long-standing open problem in astrophysics for more than a century.  In this article, we focus on a reduced model of the nonlinear Jeans instability in an expanding Newtonian universe, which is described by a class of second-order nonlinear hyperbolic equations.
 \begin{equation*}
 	\Box \varrho(x^\mu) +\frac{\mathcal{a} }{t} \del{t}\varrho(x^\mu) -
 	\frac{\mathcal{b}}{t^2} \varrho(x^\mu) (1+  \varrho(x^\mu) ) -\frac{\mathcal{c}-\ck}{1+\varrho(x^\mu)} (\del{t}\varrho(x^\mu))^2=  \ck F(t).
 \end{equation*}
We establish a family of nonlinear self-increasing blowup solutions (where the solution itself becomes infinite in a stable ODE-type blowup) for this equation. Furthermore, we provide estimates on the growth rate of $\varrho$, which may help explain why the nonlinear structures in the universe grow much faster in astrophysical observations than predicted by the classical Jeans instability.

 \vspace{2mm}


\vspace{2mm}

\end{abstract}

	\maketitle
	
	\setcounter{tocdepth}{2}
	
	\pagenumbering{roman} \pagenumbering{arabic}

\section{Introduction}
Gravitational instability is extremely important in astrophysics, as it governs the mass accretion of self-gravitating systems and contributes to understanding the formation of stellar systems and nonlinear structures in the universe. 
However, up to now, gravitational instability has only been studied in the \textit{linear regime}. The first linearized gravitational instability was studied by Jeans \cite{Jeans1902} in 1902 for Newtonian gravity (hence also known as the \textit{Jeans instability}) and was later generalized to general relativity by Lifshitz \cite{Lifshitz1946} in 1946. 
With the recognition of cosmic expansion, the Jeans instability was extended to an expanding background universe by Bonnor \cite{Bonnor1957} (see also \cite{Zeldovich1971,ViatcehslavMukhanov2013}).  As mass accumulates and density increases, the linear Jeans instability approximation becomes invalid, as the growing density causes the system to deviate significantly from the linear regime.
Moreover, the growth rate of density predicted by the classical linearized Jeans instability cannot account for the large inhomogeneities observed in the present universe or the formation of galaxies, as it is too slow and inefficient.
In fact, \cite{Bonnor1957,Zeldovich1971,ViatcehslavMukhanov2013} and our previous work \cite{Liu2022} indicate that, in expanding Newtonian models, the growth rate of density for the linear or certain partially nonlinear Jeans instability follows an order of $\sim t^{\frac{2}{3}}$. Due to the incompleteness of these astrophysical theories and the many mysterious phenomena potentially related to nonlinear Jeans instability (e.g., space charges in conductors), we must undertake a nonlinear analysis of Jeans instability.  However, as pointed out by Rendall \cite{Rendall2002} in 2002, no results are available on the fully nonlinear case of Jeans instability, which has remained a long-standing open problem in astrophysics.

In this article, we focus on a ``reduced model'' of the nonlinear Jeans instability for expanding Newtonian universes, which can be characterized by solutions exhibiting nonlinear self-increasing blowup\footnote{The solution itself becomes infinite at some time, known as \textit{ODE-type blowup}'' according to Alinhac \cite{Alinhac1995}. } in the following class of second-order nonlinear hyperbolic equations,
 \begin{gather}
    \Box \varrho(x^\mu) +\frac{\mathcal{a} }{t} \del{t}\varrho(x^\mu) -
	\frac{\mathcal{b}}{t^2} \varrho(x^\mu) (1+  \varrho(x^\mu) ) -\frac{\mathcal{c}-\ck}{1+\varrho(x^\mu)} (\del{t}\varrho(x^\mu))^2=  \ck F(t) , \label{e:maineq0}   \\
	\varrho|_{t=t_0}= \mathring{\varrho}(x^i)>0  \AND 	\del{t} \varrho|_{t=t_0}=   \mathring{\varrho}_0(x^i) >0 , \label{e:maineq1}
\end{gather}
where\footnote{In this article, we use the index convention given in \S\ref{s:AIN}, i.e., $\mu=0,\cdots,n$ and $i=1,\cdots,n$, $x_0=t$. } for simplicity, we define $\Box:=\partial^2_t-\Delta_g=\partial^2_t-\mathcal{g}^{ij}(t) \del{i} \del{j}$, where $\mathcal{g}^{ij}(t)$ is a prescribed Riemannian metric. We set $x^\mu:=(t,x^1,\cdots,x^n)\in[t_0,\infty)\times \Tbb^n$ with $t_0>0$, and assume that $\mathring{\varrho}(x^i)$ and $\mathring{\varrho}_0(x^i)$ are given positive-valued functions. The constants $\ca, \cb, \cc, \ck$ satisfy 
\begin{equation}\label{e:abcdk}
	\ca>1, \quad \cb>0, \quad 1<\cc < 3/2  \AND 3 c-\sqrt{2} \sqrt{8 \cc-5}<\ck<3 c+\sqrt{2} \sqrt{8 \cc-5} .
\end{equation}
In general, $\mathcal{g}^{ij}$ and $F$ are typically allowed to depend on the spatial variables $x^i$. However, to simplify the model and capture key velocity field features in fully nonlinear Jeans instability, we focus on the following specific functions, $\cg^{ij}(t)$ and $F(t)$, 
\begin{equation}\label{e:Fdef}
	\cg^{ij}(t):= \frac{\mathcal{m}^2 (\del{t} f(t))^2}{(1+f(t))^2}\delta^{ij}  \AND F(t):= \frac{(\del{t} f(t))^2}{1+f(t)},
\end{equation}
where $\cm\in \Rbb$ is a prescribed constant, and $f(t)$ satisfies the following ordinary differential equation (ODE),
\begin{gather}
	\partial^2_t f(t)+\frac{\ca}{t} \del{t} f(t)-\frac{\cb}{t^2} f(t)(1+  f(t))-\frac{\cc}{1+f(t)}(\del{t} f(t))^2=   0 , \label{e:feq0}\\
	f(t_0)= \mf>0 \AND
	\del{t}f(t_0)=   \mf_0>0.  \label{e:feq1}
\end{gather}
where $\mf,\mf_0>0$ are positive constants.

As shown in our companion article \cite{Liu2022b}, readers can \textit{see the derivation of the model equations \eqref{e:feq0}--\eqref{e:feq1}} (see \cite[eq. $(35)$]{Liu2022b}) from the nonlinear Euler--Poisson system by considering homogeneous and isotropic perturbations. 
As for the model \eqref{e:maineq0}--\eqref{e:maineq1}, its derivation requires additional calculations and Helmholtz decompositions of the spatial derivatives of velocity, following the approach in \cite{Liu2022b} or Peebles  \cite[\S II.9]{Peebles2020}. However, one usually cannot derive this model directly from the nonlinear Euler--Poisson system; instead, one obtains significantly more complex forms that depend on the initial distributions of density and velocity. These forms are coupled with velocity fields and typically involve more intricate expressions for $\mathcal{g}^{ij}$ and $F$, which depend on the velocity field. We omit the lengthy details and leave them for future work.  We simplify the complexity of $\mathcal{g}^{ij}$ and $F$ by making specific choices and selecting the model \eqref{e:maineq0}--\eqref{e:maineq1}. Hence, we refer to this as a ``reduced model''. The functions $\cg^{ij}(t)$ and $F(t)$ given by \eqref{e:Fdef} may not accurately represent the fully nonlinear Jeans instability. However, we use them at this stage to explore fundamental ideas. Therefore, as in \cite{Liu2022b}, if we set
\begin{equation}\label{e:spabc}
	\ca=4/3, \quad \cb=2/3, \quad \cc=4/3 \AND  3 c-\sqrt{2} \sqrt{8 \cc-5}<\ck<3 c+\sqrt{2} \sqrt{8 \cc-5},
\end{equation}
then the system \eqref{e:maineq0}--\eqref{e:maineq1} and \eqref{e:feq0}--\eqref{e:feq1} can be regarded as a \textit{poor man’s model} of the nonlinear Jeans instability of the expanding Newtonian universe for the Euler--Poisson system, while ignoring the velocity field equations and tidal forces. Certain effects of the velocity field are captured by the source term $F(t)$ and $\cg^{ij}(t)$. The fully nonlinear Jeans instability of the expanding Newtonian universe for the Euler--Poisson equations is in preparation. In astrophysical literature and \cite{Liu2022b}, the variable $\varrho$ represents the relative density of the fluids. Thus, an increase in the relative density $\varrho$ indicates mass accretion and the formation of nonlinear structures in the universe and stellar systems.

\begin{remark}[Further relations between physical Jeans instability and the reduced model]
	During the review process of this article, we have obtained two new results  \cite{Liu2023a, Liu2024a} based on this work. Below, we explain the relationship between this reduced model and physical systems in light of these new findings and insights.
	From a mathematical perspective, the parameter range for $\ca,\cb,\cc$ in \eqref{e:abcdk} is determined by the analysis in this article. However, as discussed above, this range \textit{encompasses the physically relevant regime}. 
	For instance, in both the Euler–Poisson and Einstein–Euler systems, after appropriate transformations and linearization (see \cite{Peebles2020,Weinberg1976,ViatcehslavMukhanov2013,Zeldovich1971,Bonnor1957}), the resulting linear equations align with the linear part of \eqref{e:maineq0}, where $\ca,\cb,\cc$  satisfy \eqref{e:spabc}. 
	Additionally, in \cite[\S II.9]{Peebles2020}, Peebles derived the nonlinear equation for the density contrast $\delta$ from the cosmological Euler–Poisson system—similar equations have also been derived and analyzed in our works \cite{Liu2022,Liu2022b,Liu2023a}. Furthermore, Noh and Hwang (see, e.g., \cite{Noh2004,Hwang2013a}) conducted a series of studies on derivations of second-order and higher-order approximations of the Einstein–Euler system in various gauges. By expanding Peebles’ equation and the Noh–Hwang equations while \textit{neglecting} tidal forces, shear effects, and rotational contributions, one obtains a second-order approximation structurally similar to \eqref{e:maineq0}. 
	In other words, \eqref{e:maineq0}, with $\ca,\cb,\cc$ satisfying \eqref{e:spabc}, captures key terms from Peebles' and Noh–Hwang’s equations.  
However, the remaining difficult terms in these equations are modeled by the source term $F(t)$, which is introduced to simplify the analysis while maintaining the correct order of the corresponding terms in the physical equations. 
This modification is one reason we refer to \eqref{e:maineq0} as a reduced model, where the source term helps to \textit{synchronize} the blow-up time. Removing this term requires additional ideas and further work, which we  address in our new article \cite{Liu2024a} and ongoing works. We will continue to explore these issues in future studies. 
\end{remark}
\begin{remark}
	Comparing the reduced model \eqref{e:maineq0} with Peebles' equation and the Noh-Hwang equations (see, e.g., \cite{Peebles2020,Noh2004,Hwang2013a,Liu2023a}), we examine how the expanding universe parameters are represented in the reduced model:  (1) the damping term coefficient $\frac{\ca}{t}$ characterizes the Hubble parameter; (2) $\mathcal{g}^{ij}$ reflects sound speeds, scale factors, and fluid velocities; (3) $\frac{\cb}{t^2}$ represents background density; while (4) $\frac{(\del{t}\varrho)^2}{1+\varrho}$ originates from fluid expansion.
\end{remark}

In conclusion, we will prove the following:

(1) The solution to the main equations \eqref{e:maineq0}–\eqref{e:maineq1}, when the initial data $\mathring{\varrho}(x^i)$ and $\mathring{\varrho}_0(x^i)$ deviate from positive constants $\mf$ and $\mf_0$ by small perturbations, undergoes nonlinear self-increasing. As a result, both $\varrho$ and $\varrho_0$ grow infinitely, leading to a blow-up.

(2) Furthermore, if the initial data also satisfies \begin{equation}\label{e:prdata} \mf_0 > \frac{(1-\ca)(1+\mf)}{(1-\cc) t_0}, \end{equation} then this self-increasing leads to a blow-up at a finite time, $t_m < \infty$ (i.e., a finite-time blow-up).

(3) Additionally, we estimate that the \textit{growth rate} of $\varrho$ is much faster than that of the classical linearized Jeans instability. Recall that the growth rate for the linear Jeans instability in an expanding Newtonian universe is of the order $\sim t^{\frac{2}{3}}$ (see \cite{Bonnor1957,Zeldovich1971,ViatcehslavMukhanov2013,Liu2022}). For example, in the special case \eqref{e:spabc}, the growth rate of $\varrho$ is at least of order $\sim \exp(t^{\frac{2}{3}})$. If the data satisfies \eqref{e:prdata}, the solution blows up at a finite time with a rate much faster than this exponential growth. We also provide (see Theorems \ref{t:mainthm0} and \ref{t:mainthm1} for detailed statements) a lower bound estimate for the growth rate in this case.

In conclusion, this reduced model suggests that the growth rate of the nonlinear Jeans instability is much faster than that of the classical (linear) instability. This is consistent with astrophysical observations and provides theoretical explanations for these observations. In other words, this article suggests that the deviations in the growth rates of $\varrho$ from the classical Jeans instability are due to nonlinear effects.

The \textit{basic idea} behind proving the ODE-type blowups (i.e., the solution itself becomes infinite at some time $t_m$ through self-increasing) is to treat the target equations \eqref{e:maineq0} as nonlinear perturbations of the ODE \eqref{e:feq0}. We aim to show that the solution to the main equation \eqref{e:maineq0} is dominated by the solution to the ODE \eqref{e:feq0} (see Theorem \ref{t:mainthm1} for precise expressions). In other words,  the solution to \eqref{e:maineq0} behaves like a small perturbation of the solution $f(t)$ to the ODE \eqref{e:feq0}, with given data functions $\mathring{\varrho}$ and $\mathring{\varrho}_0$ perturbed slightly around $\mf$ and $\mf_0$. This approach indicates the basic blowup mechanism of \eqref{e:maineq0}. To achieve this, we use the Cauchy problem approach for Fuchsian systems, as introduced by Oliynyk \cite{Oliynyk2016a} (see \eqref{e:model1}–\eqref{e:model2} in Appendix \ref{s:fuc} for details).

On the other hand, following this approach,  to understand the behavior of $\varrho$ and determine its growth rate, we must first characterize the reference solution $f$ to the ODE  \eqref{e:feq0}--\eqref{e:feq1}. However, this ODE system cannot be easily or directly solved. Therefore, we estimate the solution's behavior by establishing upper and lower bounds for both  $f$ and $\partial_t f$ using suitable increasing functions. These bounds consequently determine the behavior of  $\varrho$ (see Theorem \ref{t:mainthm0} for precise statements).

In the following, we first state the main theorems of this article, then briefly outline their proofs and the structure of the paper. After that, we provide a brief review of relevant works and establish the notations used throughout this article.

\subsection{Main theorems}\label{s:mthm}
In this section, we present the main theorems of this article, stated in Theorems \ref{t:mainthm0} and \ref{t:mainthm1}. As mentioned earlier, in order to understand the behavior of $\varrho$, we first characterize the reference solution $f$. Therefore, Theorem \ref{t:mainthm0} provides a description of the reference solution $f$ to equations \eqref{e:feq0}--\eqref{e:feq1}. Based on this, we then state Theorem \ref{t:mainthm1}, which describes the solution $\varrho$ to the main equation \eqref{e:maineq0}--\eqref{e:maineq1} as a perturbation of $f$.

From now on, to simplify notation (with parameter domains $\ca, \cb, \cc$ as in \eqref{e:abcdk}), we define 
\begin{equation}\label{e:deftr}
	\triangle:=\sqrt{(1-\ca)^2+4\cb}>-\ba, \quad \ba=1-\ca<0, \quad \bc=1-\cc<0  , 
\end{equation}
and introduce the constants $\mathtt{A}$, $\mathtt{B}$, $\mathtt{C}$, $\mathtt{D}$, and $\mathtt{E}$, which depend on the initial data $\mf$ and $\mf_0$ for \eqref{e:feq0}--\eqref{e:feq1} and the parameters $\ca$, $\cb$, and $\cc$:
\begin{align*}
	\mathtt{A}:=&\frac{t_0^{   -\frac{\ba-\triangle}{2}  }}{\triangle}\biggl(  \frac{t_0   \mf_0}{(1+\mf)^2} - \frac{\ba+\triangle}{2}  \frac{\mf  } {1+\mf} \biggr), \\
	\mathtt{B}:= &  \frac{t_0^{-\frac{\ba+\triangle}{2} }}{\triangle} \Bigl( \frac{\ba-\triangle}{2}  \frac{\mf }{1+\mf}  -\frac{t_0 \mf_0}{(1+\mf)^2} \Bigr)<0, \\
	\mathtt{C}:= & \frac{2 } {2+\ba+\triangle } \left(\max\left\{\frac{\ba+\triangle}{2 \cb },1\right\}\right)^{-1} \Bigl( \ln( 1+\mf) +\frac{\ba+\triangle}{2\cb} \frac{t_0\mf_0}{1+\mf}\Bigr)t_0^{-\frac{\ba+\triangle}{2 }} >0,  \\
	\mathtt{D}:= &
	\frac{ \ba+\triangle   }{2+\ba+\triangle}  \Bigl(  \ln( 1+\mf)  - \frac{1 } { \cb} \frac{t_0\mf_0}{1+\mf}\Bigr) t_0,   \\
	\mathtt{E}:=  &
	\frac{\bc  \mf_0 t_0^{1-\ba}  }{  \ba  (1+\mf) } >0.
\end{align*}
We define two critical times $t_\star$ and $t^\star$.

\begin{definition}\label{t:tdef}
	Suppose $\mathtt{A},\;\mathtt{B}, \;\mathtt{E},\;\ba$ and $\triangle$ are defined as above, then
	\begin{enumerate}
		\item
		Let $\mathcal{R}:=\{t_r>t_0 \;|\;\mathtt{A} t_r^{\frac{\ba-\triangle}{2} } + \mathtt{B} t_r^{\frac{\ba+\triangle}{2} } + 1 =0\}$. If $\mathcal{R}$ has a minimum, we define $t_\star:=\min \mathcal{R}$.
		\item If $t_0^{\ba}> \mathtt{E}^{-1}$, we define $t^\star :=    (t_0^{\ba}- \mathtt{E}^{-1} )^{1/\ba}\in(0,\infty)$, i.e.,  $t=t^\star$ solves $1-\mathtt{E}  t_0^{\ba} +  \mathtt{E}   t^{\ba}=0$.
	\end{enumerate}
\end{definition}

With these definitions established, we now state the main theorems.

	\begin{theorem}\label{t:mainthm0}
		Suppose the constants $\ca$, $\cb$ and $\cc$ are defined by  \eqref{e:abcdk}, and $t_\star$ and $t^\star$ are defined in Definition \ref{t:tdef}. Let the initial data satisfy  $\mf, \mf_0>0$. Additionally, define $f_0(t):=\del{t} f(t)$. Then
		\begin{enumerate}
			\item $t_\star:=\min \mathcal{R} \in[0,\infty)$ exists and $t_\star>t_0$;
			\item there is a constant $t_m\in [t_\star,\infty]$, such that there is a unique solution $f\in C^2([t_0,t_m))$ to the equation \eqref{e:feq0}--\eqref{e:feq1}, and
			\begin{equation}\label{e:limf}
				\lim_{t\rightarrow t_m} f(t)=+\infty \AND \lim_{t\rightarrow t_m} f_0(t)=+\infty .
			\end{equation}
		\item  the function $f$ satisfies the following upper and lower bound estimates
		\begin{align*}
			1+f(t)>&\exp \bigl( \mathtt{C} t^{\frac{\ba+\triangle}{2} }  +\mathtt{D}  t^{-1}\bigr)      &&\text{for}\quad t\in(t_0,t_m);
			\\
			1+f(t) < & \bigl(\mathtt{A} t^{\frac{\ba-\triangle}{2} } + \mathtt{B} t^{\frac{\ba+\triangle}{2} } + 1 \bigr)^{-1}    && \text{for}\quad t\in(t_0,t_\star).
		\end{align*}
		\end{enumerate}		
		Furthermore, if the initial data additionally satisfies \eqref{e:prdata}, i.e.,
		$\mf_0 >  \ba(1+\mf) /(\bc t_0 )$,
		then
		\begin{enumerate}
			\setcounter{enumi}{3}
		\item  both	$t_\star$  and $t^\star$  exist and are finite, satisfying $t_0<t_\star<t^\star<\infty$;
		\item there exists a finite time $t_m\in [t_\star,t^\star)$, such that there is a solution $f\in C^2([t_0,t_m))$ to  equation \eqref{e:feq0} with   initial data \eqref{e:feq1},   and 	\begin{equation*}
			\lim_{t\rightarrow t_m} f(t)=+\infty \AND \lim_{t\rightarrow t_m} f_0(t)=+\infty .
		\end{equation*}
	\item the solution $f$ satisfies an improved lower bound estimate for $t\in(t_0,t_m)$,
	\begin{equation*}\label{e:ipvest}
		(1+\mf)  \bigl(1-\mathtt{E}  t_0^{\ba} +  \mathtt{E}   t^{\ba} \bigr)^{1/\bc}  < 1+f(t) .
	\end{equation*}
		\end{enumerate}
	\end{theorem}

\begin{theorem}\label{t:mainthm1}
	Suppose $s\in \Zbb_{\geq \frac{n}{2}+3}$, and let  $\ca, \cb, \cc, \ck$ be constants satisfying  \eqref{e:abcdk}.  Assume that $f\in C^2([t_0,t_m))$ is the solution to \eqref{e:feq0}--\eqref{e:feq1} given by Theorem \ref{t:mainthm0}, where $\mf>0$ and $\mf_0>0$ are prescribed. Furthermore, assume that $t_m>t_0$ so that $[t_0,t_m)$ is the maximal interval of existence for $f$ as stated in Theorem \ref{t:mainthm0}. 
	Then, there exist small constants  $\sigma_\star,\sigma>0$, such that if the initial data satisfies
	\begin{equation}\label{e:mtdata}
		\Bigl\|\frac{\mathring{\varrho}}{\mf}-1\Bigr\|_{H^s(\Tbb^n)}+  	\Bigl\|\frac{\mathring{\varrho}_0}{\mf_0}-1\Bigr\|_{H^s(\Tbb^n)}+  	\Bigl\|\frac{\cm  \mathring{\varrho}_i}{1+\mf}\Bigr\|_{H^s(\Tbb^n)} \leq \frac{1}{2} \sigma_\star\sigma ,
	\end{equation}
	then there is a  solution $\varrho\in C^2([t_0,t_m)\times \Tbb^n)$ to the equation \eqref{e:maineq0}--\eqref{e:maineq1} satisfying the estimate
	\begin{equation}\label{e:mainest}
		\Bigl\|\frac{\varrho(t)}{f(t)}-1\Bigr\|_{H^s(\Tbb^n)}+  	\Bigl\|\frac{\del{t}\varrho(t)}{f_0(t)}-1\Bigr\|_{H^s(\Tbb^n)}+ 	\Bigl\|\frac{\cm \del{i}\varrho(t)}{1+f(t)}\Bigr\|_{H^s(\Tbb^n)}\leq   C \sigma<1
	\end{equation}
	for $t\in[t_0,t_m)$, where $C>0$ is some constant.
	Moreover, $\varrho$  blows up at $t=t_m$, i.e.,
	\begin{equation*}
		\lim_{t\rightarrow t_m} \varrho(t,x^i)=+\infty \AND \lim_{t\rightarrow t_m} \varrho_0(t,x^i)=+\infty ,
	\end{equation*}
	with the rate estimates $(1-C\sigma)f \leq  \varrho \leq (1+C\sigma)f$ and $(1-C\sigma)f_0 \leq \varrho_0 \leq (1+C\sigma)f_0$ 	for $t\in[t_0,t_m)$.
\end{theorem}

\subsection{Overviews and outlines}

According to the basic idea mentioned above, this article is organized into two main sections:

In \S\ref{t:refsol}, we have two main objectives. The first is to study the reference solution $f(t)$ of the ODE system \eqref{e:feq0}--\eqref{e:feq1} in detail and to prove Theorem \ref{t:mainthm0}. We establish upper and lower bounds for $f(t)$ and its derivative $\del{t} f(t)$ using certain increasing functions. The second objective is to introduce three special functions, $g(t)$, $\chi(t)$, and $\xi(t)$ (defined in \eqref{e:gdef}, \eqref{e:Gdef}, and \eqref{e:xidef}, respectively), and to analyze their key properties, such as bounds and limits as $t$ approaches $t_m$, where $t_m$ is the maximal time of existence for $f$, as stated in Theorem \ref{t:mainthm0}. 
We study these functions because they play a crucial role in \S\ref{s:fuchian} when rewriting the main equation \eqref{e:maineq0} into its Fuchsian form (see \eqref{e:model1}--\eqref{e:model2} in Appendix \ref{s:fuc}). For instance, $g(t)$ defines the time transformation $\tau = -g(t)$, which compactifies the time domain from $[t_0,t_m)$ to $[-1,0)$, mapping the maximal time $t = t_m$ to $\tau = 0$. The Fuchsian system is then formulated using this compactified time coordinate. Similarly, $\chi(t)$ and $\xi(t)$ are essential for distinguishing singular terms from regular ones in the Fuchsian formulation—specifically, the singular term $\frac{1}{t} \textbf{B}(t,x,u) \textbf{P}u$ and the regular term $H(t,x,u)$ in the Fuchsian system \eqref{e:model1}--\eqref{e:model2}.

In \S\ref{s:fuchian}, we begin analyzing the main equation \eqref{e:maineq0}, or equivalently, the nonlinear perturbation equation obtained by subtracting \eqref{e:feq0} from \eqref{e:maineq0}, to prove Theorem \ref{t:mainthm1}. Our approach consists of seven steps. As mentioned earlier, the fundamental idea is to apply the Cauchy problem method for Fuchsian systems, as outlined in Appendix \ref{s:fuc}. The first challenge is to rewrite this perturbation equation into a Fuchsian formulation (see Steps $1$–$3$ in \S\ref{s:fuchian}). As noted in our previous works \cite{Liu2018b,Liu2018,Liu2022a,Liu2018a,Liu2022}, one of the key and most difficult steps in this process is introducing appropriate new variables (i.e., \textit{suitable Fuchsian fields}) and \textit{suitable  time compactifications}. 
In this article, we find that the most crucial and challenging construction is to define the compactified time as follows:  \begin{align}\label{e:ttf1}
	\tau := -g(t)=& -\exp\Bigl(-A\int^t_{t_0} \frac{f(s)(f(s)+1)}{s^2f_0(s)} ds\Bigr)  \notag  \\
	= & -\Bigl(1+ \cb B \int^t_{t_0} s^{\ca-2} f(s)(1+f(s))^{1-\cc}  ds \Bigr)^{-\frac{A}{\cb}}\in[-1,0),
\end{align} as given in \eqref{e:ttf}, where $A\in(0,2\cb/(3-2\cc))$ and $B:= (1+\mf)^\cc/( t_0^{ \ca} \mf_0)$. Additionally, we introduce the Fuchsian fields by weighting the perturbation variables with appropriate increasing functions (see \eqref{e:u}–\eqref{e:ui}). With this setup, the main equation \eqref{e:maineq0} eventually transforms into a Fuchsian system, which is verified in Step $4$. 
In Steps $5$ and $6$, by choosing the initial data at $\tau = -1$ to be sufficiently small, standard energy estimates ensure that the local solution can be extended for a certain time. We then apply the global existence theorem for Fuchsian systems (Theorem \ref{t:fuc}) and obtain a global solution for $\tau \in [-1, 0)$, which corresponds to $t \in [t_0, t_m)$.
Finally, by transforming the Fuchsian fields back to the original variables $(\varrho, \del{t} \varrho, \del{i} \varrho)$, we establish the self-increasing blowup behavior of $\varrho$ and $\del{t} \varrho$, thereby completing the proof.

\subsection{Related works}

The blowup problem has been widely studied for various hyperbolic equations, particularly for wave-type equations. For example, Speck \cite{Speck2020, Speck2017} studied stable ODE-type blowup results for quasilinear wave equations featuring a Riccati-type derivative-quadratic semilinear term and a form of stable Tricomi-type degeneracy formation in a specific wave equation. We refer readers to \cite{Alinhac1995, Kichenassamy2021} for an introduction to blowups in various hyperbolic systems. We also recommend consulting Speck \cite{Speck2020, Speck2017} for detailed reviews of various blowup results and their basic proof techniques.

The key idea for the global solutions presented in this article is the Cauchy problem for Fuchsian systems. The approach of using Fuchsian formulations to solve the global problem was first introduced by Oliynyk \cite{Oliynyk2016a} and later developed and applied to various problems in a series of works, such as \cite{Ames2022, Ames2022a, Beyer2021, Beyer2020, Beyer2025, Fajman2021, Fajman2023, Fajman2025, Beyer2024, Beyer2024b, Beyer2020a, Oliynyk2021, Oliynyk2024, Liu2018b, Liu2018, Liu2022a, Liu2022, Liu2022b, Liu2018a, Liu2024a, Fournodavlos2024, Marshall2023}.

In addition, we note that recent works have focused on gravitational collapse \cite{Guo2018, Guo2021} and the formation of implosion singularities \cite{Merle2022}, both of which lead to mass accretion and may have connections to the Jeans instability.

\subsection{Notations}\label{s:AIN}
Unless stated otherwise, we will apply the following conventions of notations throughout this article without recalling their meanings in the following sections.

\subsubsection{Indices and coordinates}\label{iandc}
Unless stated otherwise, our indexing convention will be as follows: we use lowercase Latin letters, e.g. $i, j,k$, for spatial indices that run from $1$ to $n$, and lowercase Greek letters, e.g. $\alpha, \beta, \gamma$, for spacetime indices
that run from $0$ to $n$.  We will follow the \textit{Einstein summation convention}, that is, repeated lower and upper indices are implicitly summed over. We use $x^{i}$ ($i=1, \cdots, n$) to denote the standard periodic coordinates on the
$n$-torus $\mathbb{T}^n$ and $t = x^{0}$ as the time coordinate on the interval $[t_0,\infty)$.

This article involves two different time coordinates: the open time $t\in[t_0,\infty)$ and the compactified time $\tau=-g(t)\in [-1,0)$, as given by \eqref{e:ttf1} (see details in \S\ref{s:stp2}). For scalar functions $f(t,x^i)$, we always use
\begin{equation} \label{e:udl}
	\underline{f}(\tau,x^i):=  f(g^{-1}(-\tau), x^i)
\end{equation}
to denote the representation of $f$ in compactified time coordinate $\tau$ throughout this article.

\subsubsection{Derivatives}
Partial derivatives with respect to the  coordinates $(x^\mu)=(t,x^i)$ will be denoted by $\partial_\mu = \partial/\partial x^\mu$. We use
$Du=(\partial_j u)$ and $\partial u = (\partial_\mu u)$ to denote the spatial and spacetime gradients, respectively, with respect to these coordinates.
We also use Greek letters to denote multi-indices, e.g.
$\alpha = (\alpha_1,\alpha_2,\ldots,\alpha_n)\in \mathbb{Z}_{\geq 0}^n$, and employ the standard notation $D^\alpha = \partial_{1}^{\alpha_1} \partial_{2}^{\alpha_2}\cdots
\partial_{n}^{\alpha_n}$ for spatial partial derivatives. It will be clear from context whether a Greek letter stands for a spacetime coordinate index or a multi-index.

\subsubsection{Function spaces, inner-products and matrix inequalities}\label{s:funsp}

Given a finite-dimensional vector space $V$, we let
$H^s(\mathbb{T}^n,V)$, $s\in \mathbb{Z}_{\geq 0}$,
denote the space of maps from $\mathbb{T}^n$ to $V$ with $s$ derivatives in $L^2(\Tbb^n)$. When the
vector space $V$ is clear from context, for example, $V=\Rbb^N$, we write $H^s(\mathbb{T}^n)$ instead of $H^s(\mathbb{T},V)$.
Letting
\begin{equation*}
	\langle{u,v\rangle} = \int_{\mathbb{T}^n} (u(x),v(x))\, d^n x,
\end{equation*}
where $(\cdot,\cdot)$
is the Euclidean inner product on $\Rbb^N$ (i.e., $(\xi,\zeta)=\xi^T\zeta$ for any $\xi, \zeta\in \Rbb^N$), denote the standard $L^2$ inner product, and the $H^s$ norm is defined by
\begin{equation*}
	\|u\|_{H^s}^2 = \sum_{0\leq |\alpha|\leq s} \langle D^\alpha u, D^\alpha u \rangle.
\end{equation*}

For matrices $A,B\in \mathbb{M}_{N\times N}$, we define
\begin{equation*}
	A\leq B \quad \Leftrightarrow \quad  (\zeta,A\zeta)\leq (\zeta,B\zeta),  \quad \forall \zeta\in \Rbb^N.
\end{equation*}


\section{The analysis of the reference solutions} \label{t:refsol}

This section accomplishes two main objectives: first, to establish Theorem \ref{t:mainthm0} concerning the reference solutions $f$ of the ODE system \eqref{e:feq0}--\eqref{e:feq1}; and second, to analyze three key functions, $g(t)$, $\chi(t)$ and $\xi(t)$, which are essential for the subsequent Fuchsian formulation.

\subsection{The time Transformation  function $g(t)$ and preliminary facts}\label{t:ttf}

We define a useful function  $g(t)$, which will appear frequently in later calculations and play an important role in the Fuchsian formulation (see \S\ref{s:fuchian}), as it serves as a \textit{compactified time transformation}. Let
\begin{align}\label{e:gdef}
	g(t):=\exp\Bigl(-A\int^t_{t_0} \frac{f(s)(f(s)+1)}{s^2 f_0(s)} ds \Bigr)>0
\end{align}
where $A\in(0,2\cb/(3-2\cc))$ is a constant.

The following lemma provides an alternative representation of $g(t)$ that involves only $f$,  without $f_0$, establishes fundamental properties of $g(t)$, and expresses $f_0$ in terms of $f$ and $g$.  These representations help eliminate $f_0$ in later calculations.

\begin{lemma}\label{t:f0fg}
	Suppose $f\in C^2([t_0,t_1))$ ($t_1>t_0$) is a solution to equations  \eqref{e:feq0}--\eqref{e:feq1}, $g(t)$ is defined by \eqref{e:gdef}, and  $f_0(t):=\del{t} f(t)$. Then
\begin{enumerate}	
	\item  $f_0$ can be expressed as
	\begin{align}\label{e:f0frl}
		f_0(t)=B^{-1} t^{-\ca} g^{-\frac{\cb}{A}}(t)(1+f(t))^{\cc} >0
	\end{align}
	for $t\in [t_0,t_1)$, where $B:= (1+\mf)^\cc/( t_0^{ \ca} \mf_0)>0 $ is a constant depending on the initial data;
	\item If the initial data satisfies $\mf>0$, then $f(t)>0$ for $t \in[t_0,t_1)$;
	\item $g(t)$ can be represented as
	\begin{equation}\label{e:gdef2}
		g (t) =\Bigl(1+ \cb B \int^t_{t_0} s^{\ca-2} f(s)(1+f(s))^{1-\cc}  ds \Bigr)^{-\frac{A}{\cb}}\in(0,1],
	\end{equation}
	for $t\in[t_0,t_1)$, with $g(t_0)=1$; 	
	\item
	$g(t)$ is strictly decreasing and  invertible on $[t_0,t_1)$.
\end{enumerate}
\end{lemma}
\begin{proof}
\underline{$(1)\&(2)$:}	Multiplying \eqref{e:feq0} by $1/(1+f)$ on both sides yields, for $t\in[t_0,t_1)$  
	\begin{equation}\label{e:f0/f}
		\del{t}\Bigl(\frac{f_0}{1+f}\Bigr)=-\frac{\ca}{t}\Bigl(\frac{f_0}{1+f}\Bigr) +\frac{\cb}{t^2} f -(1-\cc) \frac{f_0^2}{(1+f)^2}.
	\end{equation}
	Multiplying  \eqref{e:f0/f} by $(1+f)/f_0$, we obtain
	\begin{equation*}
		\del{t}\ln \Bigl(\frac{f_0}{1+f}\Bigr) =-\frac{\ca}{t}+\frac{\cb}{t^2} \frac{f(1+f)}{f_0}+(\cc-1)\frac{f_0}{1+f}.
	\end{equation*}
Using the identity
	\begin{equation*}
		\frac{f_0}{1+f}=\frac{\del{t}f}{1+f}=\del{t}\ln(1+f),
	\end{equation*}
	we obtain
	\begin{equation*}
		\del{t}\ln \Bigl(\frac{f_0}{1+f}\Bigr) =-\frac{\ca}{t}+\frac{\cb}{t^2} \frac{f(1+f)}{f_0}+\del{t}\ln(1+f)^{\cc-1}.
	\end{equation*}
	This leads to
	\begin{equation}\label{e:dtlnf0f}
		\del{t}\ln \Bigl(\frac{f_0}{(1+f)^\cc}\Bigr)  =-\frac{\ca}{t}+\frac{\cb}{t^2} \frac{f(1+f)}{f_0}.
	\end{equation}

Integrating \eqref{e:dtlnf0f} and using the initial data \eqref{e:feq1}, along with the definition  \eqref{e:gdef}, we obtain
	\begin{equation}\label{e:f0f}
		\frac{f_0(t)}{(1+f(t))^\cc}=\frac{\mf_0}{(1+\mf)^\cc} \Bigl(\frac{t}{t_0}\Bigr)^{-\ca} \exp\Bigl( \cb \int^{t}_{t_0}\frac{f(s)(1+f(s))}{s^2 f_0(s)} ds\Bigr)
		=B^{-1} t^{-\ca} g^{-\frac{\cb}{A}}(t)\geq 0 ,
	\end{equation}
	 for $t\in[t_0,t_1)$. Hence, we conclude $f_0(t)=B^{-1} t^{-\ca} g^{-\frac{\cb}{A}}(t)(1+f(t))^{\cc}$.

Since $f_0=\del{t}f$,  using \eqref{e:f0f} gives an estimate for $t\in[t_0,t_1)$,
\begin{equation*}
	\del{t}\Bigl[\frac{1}{1-\cc}(1+f)^{1-\cc}\Bigr] \geq 0 . 
\end{equation*}
Integrating this inequality and noting that $\cc>1$,  we obtain $(1+f)^{\bc}\leq (1+\mf)^{\bc}$  for $t\in[t_0,t_1)$. Since $f\in C^2([t_0,t_1))$ is the solution to \eqref{e:feq0}--\eqref{e:feq1}, we prove $f(t)>0$ for any $t\in[t_0,t_1)$ by contradictions. If there is a time $t_s\in[t_0,t_1)$ such that $f(t_s) \in[-1,0]$, then $0\leq (1+f(t_s))^{\bc}\leq (1+\mf)^{\bc}$ where $\bc<0$, which implies $f(t_s)>\mf>0$, leading to a contradiction. Thus, $f(t)\notin [-1,0]$ for any $t\in[t_0,t_1)$. On the other hand, if there is a time $t_r\in[t_0,t_1)$ such that $f(t_r) <-1$. By the continuity of $f$ (since $f\in C^2$) and $\mf>0$, then there is a time $t_s\in(t_0,t_r)$, such that $f(t_s)\in[-1,0]$. This contradicts with the fact $f(t)\notin [-1,0]$ for any $t\in[t_0,t_1)$. Hence, $f(t)>0$ for all $t\in[t_0,t_1)$.
Therefore, using \eqref{e:f0frl}, we further conclude $f_0(t) > 0$ and $f(t) > \mf>0$ for $t > t_0$.

\underline{$(3)\&(4)$:}	We now prove \eqref{e:gdef2}. 	Differentiating \eqref{e:gdef} with respect to $t$ yields
	\begin{equation*}
		\del{t}g(t)=-A g(t)   \frac{f(t)(f(t)+1)}{t^2f_0(t)}.
	\end{equation*}
Substituting $f_0$ from \eqref{e:f0frl}, we obtain 
	\begin{equation}\label{e:dtg0}
		\del{t}g(t) =-A  B g^{\frac{\cb}{A}+1}(t)  t^{\ca-2} f(t) (1+f(t))^{1-\cc} .
	\end{equation}
	This leads to
	\begin{align}\label{e:dtgf}
		\del{t} g^{-\frac{\cb}{A}}(t) = -\frac{\cb}{A}g^{-\frac{\cb}{A}-1}\del{t} g = \cb B t^{\ca-2} f(t)  (1+f(t))^{1-\cc}.
	\end{align}
	Integrating   \eqref{e:dtgf} and noting that $g(t_0)=1$ (by \eqref{e:gdef}) and $f(t)>0$ for $t>t_0$, we obtain
	\begin{align}\label{e:gba}
		g^{-\frac{\cb}{A}}(t)
		=1+ \cb B \int^t_{t_0} s^{\ca-2} f(s)(1+f(s))^{1-\cc}  ds \geq 1
	\end{align}
	for $t\in[t_0,t_1)$, which implies \eqref{e:gdef2}.
	Since $ s^{\ca-2} f(s)(1+f(s))^{1-\cc} \in(0,\infty)$ is continuous for $s\in[t_0,t_1)$, it follows from \eqref{e:gba} that  $g^{-\frac{\cb}{A}}(t)$ is strictly increasing  for $t\in[t_0,t_1)$. Consequently, $g(t)$ strictly decreasing,  invertible,  and satisfies $g(t)\in (0,1]$ for $t\in[t_0,t_1)$.
	The proof is complete.
\end{proof}


\subsection{A priori estimates for $f(t)$ and the first-order system}

For the ODE \eqref{e:feq0}--\eqref{e:feq1}, we claim the following Proposition \ref{t:bksl}, which provides a priori estimates for the solution $f(t)$. To maintain the coherence of this article, we postpone the lengthy proof of Proposition \ref{t:bksl} to \S \ref{s:pfPro}. Before proceeding, we first rewrite the ODE \eqref{e:feq0} as a first-order system for later use.

To achieve this, we introduce a variable transformation:
\begin{equation}\label{e:ydef}
	y(t):=1+f(t).
\end{equation}
Substituting this into \eqref{e:feq0}, we obtain
\begin{align}\label{e:yeq}
	\partial^2_t y(t)+\frac{\ca}{t} \del{t} y(t)-\frac{\cb}{t^2} (y(t)-1) y(t) -\frac{\cc}{y(t)}(\del{t} y(t))^2=0 .
\end{align}

By Lemma \ref{t:f0fg}, the solution $f\in C^2$ to \eqref{e:feq0}--\eqref{e:feq1} must satisfy $f(t)>0$ and $f_0(t)>0$ for $t>t_0$ (i.e., $y(t)>1$ and $y_0(t):=\del{t}y(t)=f_0(t)>0$ for $t>t_0$). Hence, we are able to introduce the following variables to rewrite  \eqref{e:feq0}. We denote
\begin{align}\label{e:q0}
 q(t):=\ln y(t), \AND  q_0(t):=\del{t}q(t) .
\end{align}
Then \eqref{e:yeq} (i.e., \eqref{e:feq0}) becomes
\begin{equation*}
	\partial^2_t q(t)+\frac{\ca}{t} \del{t} q(t) -\frac{\cb}{t^2} (e^{q(t)}-1) +(1-\cc) (\del{t} q)^2=0.
\end{equation*}
Furthermore, we rewrite this equation in terms of  $(q, q_0)$ as a first-order system,
\begin{align}		
	\del{t} q(t)= & q_0(t),  \label{e:q1}  \\
	\del{t} q_0(t)=& -\frac{\ca}{t} q_0(t)+\frac{\cb}{t^2} (e^{q(t)}-1) -(1-\cc) q_0^2. \label{e:q2}
\end{align}
By \eqref{e:q0}--\eqref{e:q2}, we can also express \eqref{e:q1}--\eqref{e:q2} in terms of $(y,q_0)$ as follows,
\begin{align}
	\del{t}y=&e^q\del{t}q=y q_0,  \label{e:upbd1} \\
	\del{t} q_0 =& -\frac{\ca}{t} q_0 +\frac{\cb}{t^2} (y-1) -(1-\cc) q_0^2.   \label{e:upbd2}
\end{align}

In the following, for the sake of simplicity in expression, we introduce two functions
\begin{equation}\label{e:LF}
	\mathcal{F}(t):=\mathtt{A} t^{\frac{\ba-\triangle}{2} } + \mathtt{B} t^{\frac{\ba+\triangle}{2} } + 1 \AND \mathcal{L}(t) :=(1+\mf)^{\bc}  \bigl(1-\mathtt{E}  t_0^{\ba} +  \mathtt{E}   t^{\ba} \bigr) .
\end{equation}
We now state the a priori estimates for the solution $f(t)$. 
 \begin{proposition}\label{t:bksl}
	Suppose $f\in C^2([t_0,t_1))$ with $t_1>t_0$ is a solution to the ODE system  \eqref{e:feq0}--\eqref{e:feq1} and  let the constants $\mathtt{A}, \mathtt{B}, \mathtt{C}, \mathtt{D}, \mathtt{E}, \ba, \bc$ and $\triangle$ be as defined in  \S\ref{s:mthm}. Then, for $t\in (t_0,t_1)$, 
	the function $f$ satisfies the estimates
	\begin{gather}
		 \exp \bigl( \mathtt{C} t^{\frac{\ba+\triangle}{2} }  +\mathtt{D}  t^{-1}\bigr)  < 1+f(t) ,
	  \label{e:fbds}\\
	  	(1+f(t))^{\bc}<  (1+\mf)^{\bc}  \bigl(1-\mathtt{E}  t_0^{\ba} +  \mathtt{E}   t^{\ba} \bigr) =\mathcal{L}(t) \label{e:fbds1} \\
		\intertext{and}
		\mathcal{F}(t)=\mathtt{A} t^{\frac{\ba-\triangle}{2} } + \mathtt{B} t^{\frac{\ba+\triangle}{2} } + 1 <(1+f(t))^{-1}  . \label{e:fbds2}
	\end{gather}
	Moreover, for $t\in (t_0,t_1)$, the derivative $f_0(t)$ is bounded by 
	\begin{equation}\label{e:foupbd}
		0<f_0(t)< -\mathtt{B}\triangle t^{\frac{\triangle+\ba}{2}-1} (1+f)^2  .
	\end{equation}
Furthermore, if $t_\star$ exists and is defined by Definition \ref{t:tdef}, then the maximal existence interval $[t_0,t_m)$ of $f$ satisfies $t_m\geq t_\star$. As a direct corollary of the above results, the function $\mathcal{L}$ remains positive at  $t=t_\star$, i.e., $\mathcal{L}(t_\star)=(1+\mf)^{\bc}  \bigl(1-\mathtt{E}  t_0^{\ba} +  \mathtt{E}   t_\star^{\ba} \bigr) >0$.
\end{proposition}


\subsection{Reference solutions and proof of Theorem \ref{t:mainthm0}}
Before proving the main theorem, Theorem \ref{t:mainthm0}, we first establish Lemma \ref{t:ttcomp}, which concerns the critical times  $t^\star$ and $t_\star$ as defined in Definition \ref{t:tdef}.
\begin{lemma}\label{t:ttcomp}
	Suppose $t_\star$ and $t^\star$ are defined as in  Definition \ref{t:tdef}. Then
	\begin{enumerate}
		\item $t_\star:=\min \mathcal{R} \in[0,\infty)$ exists and $t_\star>t_0$;
		\item Furthermore, if the initial data satisfies \eqref{e:prdata}, i.e.,
		$\mf_0 >  \ba(1+\mf) /(\bc t_0 )$, then $t^\star$ exists and satisfies   $t_0<t_\star<t^\star$.
	\end{enumerate}
\end{lemma}

\begin{proof}
$(1)$	 We first observe that  $
	\mathcal{F}(t):=\mathtt{A} t^{\frac{\ba-\triangle}{2} } + \mathtt{B} t^{\frac{\ba+\triangle}{2} } + 1$
 is continuous on $[t_0,\infty)$. A straightforward calculation gives  $\mathcal{F}(t_0)
	= 1/(1+\mf)>0 $ and $\lim_{t\rightarrow \infty}\mathcal{F}(t)=-\infty$ due to $(\ba-\triangle)/2<0$ and $(\ba+\triangle)/2>0$. Therefore, there must exist at least one time  $t_r\in(t_0,\infty)$ such that $\mathcal{F}(t_r)=\mathtt{A} t_r^{\frac{\ba-\triangle}{2} } + \mathtt{B} t_r^{\frac{\ba+\triangle}{2} } + 1 =0$, which implies that the set  $\mathcal{R}\neq \emptyset$ (recall that $\mathcal{R}$ is defined in Definition \ref{t:tdef}). Since $t_r>t_0$ for all $t_r\in \mathcal{R}$,  then there is an infimum $t_{\text{inf}}$ of $\mathcal{R}$.

	We first claim $t_{\text{inf}}> t_0$, otherwise, if $t_{\text{inf}}= t_0$, then $\mathcal{F}(t_{\text{inf}})=\mathcal{F}(t_0)>0$. Due to the definition of $t_{\text{inf}}$, for any $n \in \mathbb{N}_{+}$ there is $t_{r,n}\in \mathcal{R}$, such that $t_{\text{inf}}<t_{r,n}<t_{\text{inf}}+1/n$. We then have a sequence $\{t_{r,n}\}$ and $\lim_{n\rightarrow \infty}t_{r,n}=t_{\text{inf}}$. According to the continuity of $\mathcal{F}$, we obtain
	\begin{align}\label{e:linmcf}
		0=\lim_{n\rightarrow \infty}\mathcal{F}(t_{r,n})=\mathcal{F}(t_{\text{inf}})>0 , 
	\end{align}
	which is a contradiction. Hence,  $t_{\text{inf}}> t_0$.

Next, we claim that $t_{\text{inf}}\in \mathcal{R}$. Suppose, for contradiction, that $t_{\text{inf}}\notin \mathcal{R}$,
	then $\mathcal{F}(t_{\text{inf}})=\mathtt{A} t_{\text{inf}}^{\frac{\ba-\triangle}{2} } + \mathtt{B} t_{\text{inf}}^{\frac{\ba+\triangle}{2} } + 1 > 0$ (if $\mathcal{F}(t_{\text{inf}})=\mathtt{A} t_{\text{inf}}^{\frac{\ba-\triangle}{2} } + \mathtt{B} t_{\text{inf}}^{\frac{\ba+\triangle}{2} } + 1 < 0$, there is at least another zero of $\mathcal{F}$ in the interval  $(t_0,t_{\text{inf}})$, which contradicts the meaning of $t_{\text{inf}}$). By the continuity of $\mathcal{F}$, we again arrive at \eqref{e:linmcf}, leading to a contradiction. Thus,  $t_{\text{inf}}\in \mathcal{R}$, meaning that  $t_{\text{inf}}=\min \mathcal{R}=t_\star$.  Consequently,  $\mathcal{F}(t_\star)=0$, $t_\star$ exists and satisfies  $t_\star>t_0$.

$(2)$ 	Firstly, since the condition \eqref{e:prdata} (i.e.,
	$\mf_0 >  \ba(1+\mf) /(\bc t_0 )$) is equivalent, by the definition of $\mathtt{E}$ in \S\ref{s:mthm}, to $\mathtt{E}^{-1}<t_0^{\ba}$, then by Definition \ref{t:tdef}, we can define $t^\star :=    (t_0^{\ba}- \mathtt{E}^{-1} )^{1/\ba}>0$. By  noting  $\ba=1-\ca<0$ and $\mathtt{E}>0$, we obtain $t^\star>t_0$. We note $\mathcal{L}(t) :=	(1+\mf)^{\bc}  \bigl(1-\mathtt{E}  t_0^{\ba} +  \mathtt{E}   t^{\ba} \bigr) $ is a decreasing function for $t>t_0$ (due to $\ba <0$),  $t^\star$ is a zero of $\mathcal{L}(t)$, i.e., $\mathcal{L}(t^\star)=0$ and note $\mathcal{L}(t_0)=(1+\mf)^{\bc}>0$.

To complete the proof, we show  $t^\star>t_\star$ by contradiction. Assume instead that $t_\star \geq t^\star$. Since $\mathcal{L}$ is decreasing function, we obtain $\mathcal{L}(t_\star)
	\leq \mathcal{L}(t^\star)=0$. However, by Proposition \ref{t:bksl}, we have a corollary that  $\mathcal{L}(t_\star)>0$, which leads to a contradiction. Therefore,  $t^\star>t_\star$, completing the proof.
\end{proof}

We are now in the position to prove Theorem \ref{t:mainthm0} for the solution $f(t)$ to \eqref{e:feq0}--\eqref{e:feq1}.
\begin{proof}[Proof of Theorem \ref{t:mainthm0}]
By Lemma \ref{t:ttcomp}, statements  $(1)$ and $(4)$ follow immediately

$(2)$ To prove $(2)$, we first establish the local existence of the solution $f$. We consider the first-order system \eqref{e:upbd1}--\eqref{e:upbd2} for $(y,q_0)$ and use Theorem \ref{t:exunqthm} (Existences and Uniqueness of ODEs), there is a small time interval $[t_0,t_1)$, such that, for $t\in [t_0,t_1)$ ($t_0<t_1<t_\star$), there is a unique $C^1$ solution $(y,q_0)$. Thus,  we have a local solution $f$ in $[t_0,t_1)$. By Proposition \ref{t:bksl}, the maximal existence interval $[t_0,t_m)$ for $f$ can be extended to a time $t_m \geq t_\star$. Hence, $f\in C^2([t_0,t_m))$ is a solution of  \eqref{e:feq0}--\eqref{e:feq1}.

If $t_m=+\infty$, by taking the limits on the both sides of \eqref{e:fbds}, we readily obtain $\lim_{t\rightarrow t_m} f(t)=\infty$. 
Next, we prove that $\lim_{t\rightarrow t_m} f(t)=\infty$ for finite  $t_m<\infty$ by contradiction. Suppose $\lim_{t\rightarrow t_m} f(t)<+\infty$. Then by \eqref{e:foupbd} in  Proposition \ref{t:bksl},  we have  $\lim_{t\rightarrow t_m} f_0(t)<+\infty$, and further
\begin{equation*}
	\lim_{t\rightarrow t_m} (y,q_0)=\lim_{t\rightarrow t_m} \Bigl(1+f,\frac{f_0}{1+f}\Bigr)<+\infty .
\end{equation*}
By the continuation principle (see Theorem \ref{t:contthm1}), we can continue the solution beyond $t_m$, which contradicts the definition of $t_m$. Therefore, $\lim_{t\rightarrow t_m} f(t)=+\infty$.

On the other hand, since, by Lemma \ref{t:f0fg}.$(4)$, $g$ is a decreasing function,  which leads to $g^{-\frac{\cb}{A}}(t) \geq 1$. Using \eqref{e:f0frl} and  \eqref{e:fbds}, we  obtain
\begin{align}
	&f_0(t)=B^{-1} t^{-\ca} g^{-\frac{\cb}{A}}(t)(1+f(t))^{\cc} \geq B^{-1} t^{-\ca} (1+f(t))^{\cc} \notag  \\
	&\hspace{2cm} \geq B^{-1} t^{-\ca} \exp [(\cc-1)\bigl( \mathtt{C} t^{\frac{\ba+\triangle}{2} }  +\mathtt{D}  t^{-1}\bigr)]  (1+f(t)), \label{e:f0est2}
\end{align}
and by taking the limit on both sides, we claim
\begin{equation}\label{e:lim2}
	\lim_{t\rightarrow t_m} f_0(t) =+\infty.
\end{equation}
We point out if $t_m<\infty$, we can use $f_0(t) \geq B^{-1} t^{-\ca} (1+f(t))^{\cc}$ directly obtain \eqref{e:lim2} and if $t_m=\infty$, we use \eqref{e:f0est2} and note the fact $\lim_{x\rightarrow \infty}(x^a/e^x)=0$ and $\lim_{t\rightarrow \infty}\exp((\cc-1)Dt^{-1})=1$ to conclude \eqref{e:lim2}. Thus, we have proved $(2)$.

$(3)$ Because $f\in C^2([t_0,t_m))$ is a solution of  \eqref{e:feq0}--\eqref{e:feq1}, by \eqref{e:fbds} in Proposition \ref{t:bksl}, $1+f(t)>\exp \bigl( \mathtt{C} t^{\frac{\ba+\triangle}{2} }  +\mathtt{D}  t^{-1}\bigr)$ for $t\in(t_0,t_m)$. By \eqref{e:fbds1} and \eqref{e:fbds2}, we recall these relations  here.  With the help of Lemma \ref{t:f0fg}.$(2)$ ($f(t)>0$ in $[t_0,t_m)$) and the definition of $t_\star$ in Definition \ref{t:tdef} (since  $\mathcal{F}(t)>0$ for $t\in[t_0,t_\star)$, otherwise, $t_\star$ is not the minimum),  we have
\begin{gather*}
	0<(1+f(t))^{\bc}<  (1+\mf)^{\bc}  \bigl(1-\mathtt{E}  t_0^{\ba} +  \mathtt{E}   t^{\ba} \bigr) =\mathcal{L}(t)  \quad \text{for}\quad t\in(t_0,t_m)  ,  \\
0<\mathcal{F}(t)=\mathtt{A} t^{\frac{\ba-\triangle}{2} } + \mathtt{B} t^{\frac{\ba+\triangle}{2} } + 1 <(1+f(t))^{-1}   \quad \text{for}\quad t\in(t_0,t_\star)  .
\end{gather*}
Next, noting that $\bc<0$, we get
\begin{gather*}
	1+f(t) > (1+\mf)  \bigl(1-\mathtt{E}  t_0^{\ba} +  \mathtt{E}   t^{\ba} \bigr)^{1/\bc}   \quad \text{for}\quad t\in(t_0,t_m) , \\
	\bigl(\mathtt{A} t^{\frac{\ba-\triangle}{2} } + \mathtt{B} t^{\frac{\ba+\triangle}{2} } + 1)^{-1} > 1+f(t) \quad \text{for}\quad t\in(t_0,t_\star)  .
\end{gather*}
This completes the proof of $(3)$.

$(5)$ Now,  let us turn to $(5)$.
By \eqref{e:fbds1} in Proposition \ref{t:bksl},  we have $	(1+f(t))^{\bc}<  (1+\mf)^{\bc}  \bigl(1-\mathtt{E}  t_0^{\ba} +  \mathtt{E}   t^{\ba} \bigr)=:\mathcal{L}(t)$ for $t\in(t_0,t_m)$. If the initial data satisfies \eqref{e:prdata}, which is equivalent to $1-\mathtt{E} t_0^{\ba}<0$, we can define $t^\star$ according to Definition \ref{t:tdef}. Now, we prove that $t_m<t^\star$. Assume, for the sake of contradiction, that $t_m \geq t^\star$. Since, by Lemma \ref{t:f0fg}.$(2)$ ($f(t)>0$ in $[t_0,t_m)$), we have  $f(t^\star)>0$ for $t^\star \in(t_0,t_m)$. This implies $	0<(1+f(t^\star))^{\bc}<  (1+\mf)^{\bc}  \bigl(1-\mathtt{E}  t_0^{\ba} +  \mathtt{E}  ( t^\star)^{\ba} \bigr)=:\mathcal{L}(t^\star)=0$, which leads to a contradiction. Therefore, $t_m<t^\star$.

$(6)$ In the end, we prove $(6)$. For $t\in(t_0,t_m)$ where $t_m<t^\star$, we have $\mathcal{L}(t)=	(1+\mf)^{\bc}  \bigl(1-\mathtt{E}  t_0^{\ba} +  \mathtt{E}   t^{\ba} \bigr)>0$, since $\mathcal{L}(t)$ is decreasing and $\mathcal{L}(t^\star)=0$. Therefore, 
\begin{align*}
	1+f(t) = y(t)
	> &
	(1+\mf)  \bigl(1-\mathtt{E}  t_0^{\ba} +  \mathtt{E}   t^{\ba} \bigr)^{1/\bc} .
\end{align*}
This completes the proof.
\end{proof}


\subsection{Proof of Proposition \ref{t:bksl}}\label{s:pfPro}
In this section, we will prove Proposition \ref{t:bksl}.
\begin{proof}
	Firstly by Lemma \ref{t:f0fg}.$(1)$ and $(2)$, we note that $f(t)>0$ and $f_0(t)>0$ for any $t\in(t_0,t_1)$.  Then using \eqref{e:q0} and \eqref{e:upbd1}, we obtain
	\begin{equation*}
	y(t)=1+f(t)>1, \quad	y_0(t)>0, \quad q(t)>0 \AND q_0(t)=y_0(t)/y(t)>0
	\end{equation*}
for $t\in(t_0,t_1)$. Next, we will prove these inequalities one by one.

	\underline{$(1)$ Proof of \eqref{e:fbds}:}
	Since $f(t)$ solves  equation \eqref{e:feq0}, by \eqref{e:ydef}--\eqref{e:q2} and introducing a new variable
	\begin{equation}\label{e:q01}
		\alpha(t):=q(t)/t>0,
	\end{equation}
in terms of $(\alpha, q_0)$,  equations \eqref{e:q1}--\eqref{e:q2} become
	\begin{align}
		\del{t}\alpha(t)=&\frac{1}{t}(q_0(t)-\alpha(t) ),   \label{e:lbeq1} \\
		\del{t} q_0(t)=&-\frac{\ca}{t} q_0(t)+\frac{\cb}{t^2}(e^{\alpha(t)t}-1)-(1-\cc)q_0^2.  \label{e:lbeq2}
	\end{align}
	By \eqref{e:lbeq1} and \eqref{e:lbeq2}, we obtain
	\begin{align}\label{e:lbeq3}
		\del{t}(\alpha + k q_0) =\frac{1}{t}(q_0(t)-\alpha(t) )- \frac{k \ca}{t} q_0(t)+\frac{k \cb}{t^2}(e^{\alpha(t)t}-1)+k (\cc-1) q_0^2
	\end{align}
	where we take
	\begin{equation*}
		k=\frac{\sqrt{(\ca-1)^2+4 \cb}-\ca+1}{2 \cb} =\frac{\ba+\triangle}{2 \cb }>0 .
	\end{equation*}
	Since $e^{\alpha(t)t}-1 > \alpha(t) t$ and $\cc>1$,  and using the fact that $1-k\ca=k(k\cb-1)$,  we derive the following lower bounds for \eqref{e:lbeq3} and \eqref{e:lbeq2}
	\begin{align}
		\del{t}(\alpha + k q_0)
		> & \frac{1}{t}(q_0(t)-\alpha(t) )- \frac{k\ca}{t} q_0(t)+\frac{k\cb}{t}\alpha(t)
		=  \frac{q_0}{t}(1-k\ca )+ \frac{\alpha}{t}(k\cb-1)
		\notag  \\
		= &  \frac{q_0}{t} k (k \cb-1 )+ \frac{\alpha}{t}(k\cb-1) = \frac{k\cb-1}{t}(\alpha+kq_0) , \label{e:lbeq5}\\
		\del{0} q_0(t) >  &  -\frac{\ca}{t} q_0(t)+\frac{\cb}{t} \alpha(t). \label{e:dtq0}
	\end{align}
Thus, \eqref{e:lbeq5} and \eqref{e:dtq0} imply
	\begin{align}
		\del{t}[t^{1-k\cb}(\alpha+kq_0)] >  0 \AND
		\del{0} (t^{\ca} q_0) >  \cb \alpha(t) t^{\ca-1}.  \label{e:dtq1}
	\end{align}
	Integrating the first differential inequality in \eqref{e:dtq1} yields, for $t\in(t_0,t_1)$,
	\begin{align}\label{e:aqineq1}
		\alpha+q_0 > (\alpha(t_0)+ k q_0(t_0))t_0^{1-k\cb} t^{k\cb-1}/\bar{k} \quad \text{where } \bar{k}:=\max\{k,1\} \geq 1.
	\end{align}
	From the definitions of $q_0$ and $\alpha$ in \eqref{e:q0} and \eqref{e:q01}, we observe that \eqref{e:aqineq1} implies 
	\begin{align*}
		\alpha+\del{t} q =\alpha+\del{t} (\alpha t)  > (\alpha(t_0)+k q_0(t_0))t_0^{1-k\cb} t^{k\cb-1}/\bar{k},
	\end{align*}
	which  in turn  implies
	\begin{align}\label{e:aqineq5}
		\del{t}(t^2\alpha) > (\alpha(t_0)+ k q_0(t_0))t_0^{1-k\cb} t^{k\cb}/\bar{k}.
	\end{align}
Integrating \eqref{e:aqineq5} and using  $\alpha(t_0) =\ln( 1+\mf)/t_0$ and $q_0(t_0) = \mf_0/(1+\mf)$, we get, for $t\in(t_0,t_1)$,
	\begin{equation*}
		\alpha(t) >  \frac{1}{(k\cb+1)\bar{k}} (\alpha(t_0)+kq_0(t_0))t_0^{1-k\cb} t^{k\cb-1}  + \frac{k\cb}{k\cb+1} \Bigl[  \alpha(t_0)   - \frac{1}{\cb}q_0(t_0) \Bigr] t_0^{2} t^{-2}=\mathtt{C} t^{\frac{\ba+\triangle}{2}-1} + \mathtt{D} t^{-2} .
	\end{equation*}
By \eqref{e:q0} and \eqref{e:q01}, we have $1+f(t)= y(t)=\exp(\alpha(t) t)$, and this, for $t\in(t_0,t_1)$, leads to
	\begin{equation*}
		1+f(t) = y(t) > \exp (\mathtt{C} t^{\frac{\ba+\triangle}{2}} + \mathtt{D} t^{-1}).
	\end{equation*}

	\underline{$(2)$ Proof of \eqref{e:fbds2}:}
	Multiplying both sides of \eqref{e:q2} by  $1/y$ yields
	\begin{equation*}
		\del{t} \Bigl(\frac{q_0}{y}\Bigr)  =  -\frac{\ca}{t} \frac{q_0}{y} +\frac{\cb}{t^2} \Bigl(1-\frac{1}{y }\Bigr) + (\cc-2) \frac{q_0^2}{y }.
	\end{equation*}
	Then, recalling that $\cc\in(1,3/2)$, and since $y>1$ and $q_0>0$, we obtain
	\begin{equation*}
		\del{t} \Bigl(\frac{q_0}{y}\Bigr) +\frac{\ca}{t} \frac{q_0}{y} < \frac{\cb}{t^2} \Bigl(1-\frac{1}{y }\Bigr) ,
	\end{equation*}
	which in turn implies
	\begin{equation}\label{e:dtq0y}
		\del{t} \Bigl(t^\ca\frac{q_0}{y}\Bigr)  <  \cb t^{\ca-2} \Bigl(1-\frac{1}{y }\Bigr)  .
	\end{equation}
	On the other hand, by \eqref{e:q0}, we have $\del{t}y=y q_0$, so direct calculations lead to
	\begin{align}\label{e:q0y}
		\del{t} \Bigl(\frac{1}{y}\Bigr) =- y^{-2} \del{t} y =-\frac{q_0}{y} .
	\end{align}

	Let
	\begin{equation}\label{e:defQY}
		\mathcal{Q} :=\frac{q_0}{y}  \AND  \mathcal{Y}:=\frac{1}{y} .
	\end{equation}
	Then, \eqref{e:dtq0y} and \eqref{e:q0y} become
	\begin{align}\label{e:YQeq}
		\del{t} (t^\ca\mathcal{Q}) < \cb t^{\ca-2}(1-\mathcal{Y}) \AND \del{t}\mathcal{Y} =-\mathcal{Q}.
	\end{align}
Using \eqref{e:YQeq}, and with direct calculations, we obtain
	\begin{equation}\label{e:dtqy2}
		\del{t} (t^\ca\mathcal{Q}+\ell t^{\ca-1}\mathcal{Y}) <
		\cb t^{\ca-2} -\ell t^{\ca-1}\mathcal{Q} + \bigl[\ell  (\ca-1)-\cb\bigr] t^{\ca-2} \mathcal{Y}
	\end{equation}
	where we take
	\begin{equation*}
		\ell:  =\frac{1-\ca-\sqrt{(1-\ca)^2+4\cb}}{2} =\frac{\ba-\triangle}{2  }<0 .
	\end{equation*}
	In this case, we have
	\begin{equation}\label{e:lbdeq}
		\ell  (\ca-1)-\cb =-\ell ^2.
	\end{equation}
Then, using \eqref{e:lbdeq}, \eqref{e:dtqy2} becomes
	\begin{equation*}
		\del{t} (t^\ca\mathcal{Q}+\ell t^{\ca-1}\mathcal{Y} ) +\ell t^{-1}(t^\ca \mathcal{Q} +\ell  t^{\ca-1} \mathcal{Y})
		< \cb t^{\ca-2}
	\end{equation*}
	which in turn leads to
	\begin{equation*}
		\del{t}  (t^{\ell  +\ca}\mathcal{Q}+\ell t^{\ell  +\ca-1}\mathcal{Y}  )  <   \cb t^{\ell  +\ca-2} .
	\end{equation*}
	Then integrating it yields, for $t\in(t_0,t_1)$,
	\begin{equation*}
		\mathcal{Q}+\ell t^{-1}\mathcal{Y}    <
		\Bigl(C_2 -\frac{\cb t_0^{\ell  +\ca-1}}{\ell  +\ca-1} \Bigr)t^{-\ell  -\ca}+  \frac{\cb}{\ell  +\ca-1} t^{ -1} \quad  \text{where} \quad C_2
		=\frac{t_0^{\ell  +\ca}\mf_0}{(1+\mf)^2}+\frac{\ell t_0^{\ell  +\ca-1}}{1+\mf}.
	\end{equation*}
By \eqref{e:defQY}, and with the help of \eqref{e:q0y}, the above inequality becomes
	\begin{equation}\label{e:QY2}
		-\del{t} \Bigl(\frac{1}{y}\Bigr)+\ell t^{-1}\frac{1}{y}  =\frac{q_0}{y}+\ell t^{-1}\frac{1}{y}  <  \Bigl(C_2 -\frac{\cb t_0^{\ell  +\ca-1}}{\ell  +\ca-1} \Bigr)t^{-\ell  -\ca}+  \frac{\cb}{\ell  +\ca-1} t^{ -1}  .
	\end{equation}
	This inequality implies
	\begin{equation*}
		\del{t} \Bigl(t^{-\ell }\frac{1}{y}\Bigr)  >    \Bigl(\frac{\cb t_0^{\ell  +\ca-1}}{\ell  +\ca-1}-C_2 \Bigr)t^{-2\ell  -\ca}  -   \frac{\cb}{\ell  +\ca-1} t^{ -1-\ell }.
	\end{equation*}
Then, integrating it for  $t\in(t_0,t_1)$,   with the help of \eqref{e:lbdeq} (i.e.,  $\ell(\ell  +\ca-1)=\cb$), we get
	\begin{align*}
		\frac{1}{y} > &  \frac{1}{1+\mf}t_0^{-\ell } t^{\ell }+  \frac{1}{1-2\ell  -\ca} \Bigl(\frac{\cb t_0^{\ell  +\ca-1}}{\ell  +\ca-1}-C_2 \Bigr) (t^{1-\ell  -\ca} -t_0^{1-2\ell  -\ca} t^{\ell }) +  (1  -t_0^{ -\ell } t^{\ell }  )  \notag  \\
		= & \biggl(-\frac{t_0^{-\ell }\mf}{1+\mf}   - \frac{1}{1-2\ell  -\ca} \Bigl( \ell  t_0^{  -\ell  }  -\frac{t_0^{ 1-\ell  }  \mf_0}{(1+\mf)^2}-\frac{\ell t_0^{   -\ell  } } {1+\mf} \Bigr) \biggr) t^{\ell }\notag  \\
		& +  \frac{1}{1-2\ell  -\ca} \Bigl( \ell  t_0^{\ell  +\ca-1} -\frac{t_0^{\ell  +\ca}\mf_0}{(1+\mf)^2}-\frac{\ell t_0^{\ell  +\ca-1}}{1+\mf} \Bigr) t^{1-\ell  -\ca} + 1 \notag  \\
		= & \frac{ t^{\ell }}{1-2\ell  -\ca}\biggl(  \frac{t_0^{ 1-\ell  }  \mf_0}{(1+\mf)^2} -  \frac{(1-\ell -\ca ) \mf t_0^{   -\ell  } } {1+\mf} \biggr)+  \frac{t^{1-\ell  -\ca} }{1-2\ell  -\ca} \Bigl(  \frac{\ell  \mf t_0^{\ell  +\ca-1}}{1+\mf}  -\frac{t_0^{\ell  +\ca}\mf_0}{(1+\mf)^2} \Bigr) + 1 \notag  \\
		= & \mathtt{A} t^{\frac{\ba-\triangle}{2} } + \mathtt{B} t^{\frac{\ba+\triangle}{2} } + 1 .
	\end{align*}
	In terms of $f$, we obtain, for $t\in(t_0,t_1)$,
	\begin{equation}\label{e:calfineq}
		(1+f)^{-1}=y^{-1} >  \mathtt{A} t^{\frac{\ba-\triangle}{2} } + \mathtt{B} t^{\frac{\ba+\triangle}{2} } + 1   .
	\end{equation}

\underline{$(3)$ Proof of \eqref{e:foupbd}:}  	
	Using \eqref{e:QY2}, together with  $f=y-1$ and \eqref{e:deftr} (recalling that $\triangle>-\ba$, so that $\ba-\triangle<2\ba<0$), we obtain,
	\begin{equation}\label{e:q0est}
		q_0   <
		\Bigl(\frac{t_0^{\ell  +\ca}\mf_0}{(1+\mf)^2}-\frac{\ell t_0^{\ell  +\ca-1} \mf}{1+\mf}   \Bigr)t^{-\ell  -\ca} y+  \ell t^{ -1} (y -1)
		= -\mathtt{B}\triangle t^{\frac{\triangle+\ba}{2}-1} y+  \frac{\ba-\triangle}{2}  t^{ -1} f<-\mathtt{B}\triangle t^{\frac{\triangle+\ba}{2}-1} y  .
	\end{equation}
Applying \eqref{e:upbd1}, we conclude that
	\begin{equation*}
		f_0=\del{t}y=y q_0<-\mathtt{B}\triangle t^{\frac{\triangle+\ba}{2}-1} (1+f)^2+   \frac{\ba-\triangle}{2}  t^{ -1} f(1+f) <-\mathtt{B}\triangle t^{\frac{\triangle+\ba}{2}-1} (1+f)^2.
	\end{equation*}

	\underline{$(4)$ Proof of \eqref{e:fbds1}:} 	
Multiplying both sides of \eqref{e:upbd2} by $1/y^{c-1}$,  we arrive at
	\begin{equation*}
		\del{t} \Bigl(\frac{q_0}{y^{\cc-1}}\Bigr) -(1-\cc) \frac{q_0^2}{y^{c-1}} =  -\frac{\ca}{t} \frac{q_0}{y^{\cc-1}} +\frac{\cb}{t^2} \bigl(y^{2-\cc}-y^{1-\cc}\bigr) -(1-\cc) \frac{q_0^2}{y^{\cc-1}},
	\end{equation*}
 which simplifies to
	\begin{equation*}
		\del{t} \Bigl(\frac{q_0}{y^{\cc-1}}\Bigr) +\frac{\ca}{t} \frac{q_0}{y^{\cc-1}} =   \frac{\cb}{t^2} \bigl(y^{2-\cc}- y^{1-\cc}\bigr) .
	\end{equation*}
This further implies
	\begin{equation}\label{e:upbd3}
		\del{t} \Bigl(t^\ca\frac{q_0}{y^{\cc-1}}\Bigr)   =   \cb t^{\ca-2}  \bigl(y^{2-\cc}- y^{1-\cc}\bigr) .
	\end{equation}
On the other hand, applying \eqref{e:upbd1}, we have
	\begin{equation}\label{e:upbd4}
		\del{t} y^{1-\cc} =(1-\cc) y^{-\cc} \del{t} y =(1-\cc) y^{1-\cc} q_0.
	\end{equation}

Define
	\begin{equation}\label{e:defXY}
		X:=q_0 y^{1-c}  \AND Y:=y^{1-c},
	\end{equation}
Substituting these into \eqref{e:upbd3} and \eqref{e:upbd4}, we obtain
	\begin{align}\label{e:dtX}
		\del{t}  (t^\ca X )  =  \cb t^{\ca-2}   (y - 1  ) Y  \AND \del{t} Y  =(1-\cc) X .
	\end{align}
	By \eqref{e:dtX},  through direct calculations and using  $f=y-1>0$, we obtain
	\begin{align*}
		\del{t}  (t^\ca X+\lambda t^{\ca-1}  Y )   =  &  [\cb (y- 1 )+\lambda(\ca-1)] t^{\ca-2} Y +\lambda (1-\cc) t^{\ca-1} X   \notag \\
		>   &  \lambda(\ca-1)t^{\ca-2} Y +\lambda (1-\cc) t^{\ca-1} X =  \lambda (1-\cc) t^{-1}  \bigl( t^{\ca} X+\lambda t^{\ca-1}Y  \bigr).
	\end{align*}
	where $\lambda$ is defined as $\lambda:=(\ca-1)/(1-\cc)=-\ba/\bc$.
	Thus, we obtain
	\begin{align*}
		\del{t}  (t^\ca X+\lambda t^{\ca-1}  Y )   -  \lambda (1-\cc) t^{-1}  \bigl( t^{\ca} X+\lambda t^{\ca-1}Y  \bigr)
	    > 0 ,
	\end{align*}
which can be rewritten as
	\begin{align*}
		\del{t} (t^{\ca-\lambda(1-\cc)}   X+\lambda t^{\ca-1-\lambda(1-\cc)}   Y )
	>  0 .
	\end{align*}
Integrating this inequality, for $t\in(t_0,t_1)$,  we obtain
	\begin{equation*}
		t^{\ca-\lambda(1-\cc)}   X+\lambda t^{\ca-1-\lambda(1-\cc)}   Y
	>   C_0 :=\frac{t_0 \mf_0}{(1+\mf)^\cc}  + \frac{\ca-1}{1-\cc} (1+\mf)^{1-\cc} .
	\end{equation*}
This further implies
	\begin{equation*}
		X+\lambda t^{-1 }   Y
	>  C_0  t^{-\ca+\lambda(1-\cc)}.
	\end{equation*}
	By \eqref{e:upbd1}, \eqref{e:defXY},  \eqref{e:q0} and noting that $\cc>1$, the above inequality becomes
	\begin{equation*}
		\del{t}  y^{1-\cc}   +(1-\cc)\lambda t^{-1 }   y^{1-\cc}
	<  C_0 (1-\cc) t^{-\ca+\lambda(1-\cc)}  ,
	\end{equation*}
	which yields
	\begin{equation*}
		\del{t} ( t^{(1-\cc)\lambda} y^{1-\cc} )
		<  C_0 (1-\cc) t^{-\ca+2\lambda(1-\cc)}  .
	\end{equation*}
	Integrating this inequality, noting that $\cc>1$ and using $\lambda =-\ba/\bc$, after lengthy calculations, we arrive at, for $t\in(t_0,t_1)$,
	\begin{align}\label{e:yc}
		(1+f(t))^{\bc}=y^{\bc}(t)<(1+\mf)^{\bc}\biggl(1-\frac{\bc t_0 \mf_0 }{\ba (1+\mf)}     +  \frac{\bc  \mf_0 t_0^{1-\ba}  }{  \ba  (1+\mf) }    t^{\ba} \biggr)	= (1+\mf)^{\bc}  \bigl(1-\mathtt{E}  t_0^{\ba} +  \mathtt{E}   t^{\ba} \bigr) .
	\end{align}
	Note that $(1+\mf)^{\bc}  \bigl(1-\mathtt{E}  t_0^{\ba} +  \mathtt{E}   t^{\ba} \bigr) $ is a  decreasing function,  and $[(1+\mf)^{\bc}  \bigl(1-\mathtt{E}  t_0^{\ba} +  \mathtt{E}   t^{\ba} \bigr) ]|_{t=t_0}=(1+\mf)^{\bc}$.

\underline{$(5)$ Extension of solutions and  $\mathcal{L}(t_\star)>0$:}  Suppose $t_\star$ is defined by Definition \ref{t:tdef} as the smallest root of $\mathcal{F}$ that is greater than $t_0$. Since $f\in C^2$ solves \eqref{e:feq0} for any $t\in [t_0,t_1)$,  it follows that,  by \eqref{e:ydef}--\eqref{e:upbd2}, $(y,q_0)$ solves the system \eqref{e:upbd1}--\eqref{e:upbd2} for any $t\in [t_0,t_1)$. Let us confirm that if $t_m$ denotes the maximal time of existence for  $(y,q_0)$, then $t_m\geq t_\star$. Suppose, on the contrary, that  $t_m< t_\star$. Since the solution $f\in C^2([t_0,t_m))$ exists, it follows from \eqref{e:calfineq} that   $(y(t))^{-1}=(1+f(t))^{-1}>\mathcal{F}(t)>0$ for $t\in (t_0,t_m)$ (recall that $\mathcal{F}(t)$ is defined by \eqref{e:LF} and satisfies $\mathcal{F}(t)>0$ for all $t\in[t_0,t_\star)$).  Thus, we have the bound   $1<y(t)<(\mathcal{F}(t))^{-1}<\infty$ for $t\in (t_0,t_m)$. Along with this bound on $y$, we also obtain, from \eqref{e:upbd1} and \eqref{e:q0est}, that $0<y_0(t)/y(t)=q_0(t)<-\mathtt{B}\triangle t^{\frac{\triangle+\ba}{2}-1} y(t)<\infty$ (recall $\mathtt{B}<0$ in \S\ref{s:mthm}) for $t\in (t_0,t_m)$. 
By the continuation principle \ref{t:contthm2}, the solution $(y,q_0)$ can be extended beyond $t=t_m$, contradicting the assumption that $t_m$ is the maximal time of existence for $(y,q_0)$. Therefore, we conclude that $t_m \geq t_\star$.  
Finally, since $f\in C^2([t_0,t_m))$ for $t_m>t_\star$ solves the ODE \eqref{e:feq0}--\eqref{e:feq1}, it follows from \eqref{e:yc} that  $(1+f(t))^{\bc} < (1+\mf)^{\bc}  \bigl(1-\mathtt{E}  t_0^{\ba} +  \mathtt{E}   t^{\ba} \bigr) =\mathcal{L}(t)$ for $t\in(t_0,t_m)$. Consequently, we obtain $\mathcal{L}(t_\star)>(1+f(t_\star))^{\bc} >0$. This completes the proof.
\end{proof}

\subsection{Estimates of two crucial quantities $\chi(t)$ and $\xi(t)$}
In this section, we  estimate two important quantities $\chi(t)$ and $\xi(t)$, which will be frequently used in  reformulating the main equation \eqref{e:maineq0}--\eqref{e:maineq1} into the Fuchsian formulation \eqref{e:model1}--\eqref{e:model2} in Appendix \ref{s:fuc}.  The first quantity is defined by
\begin{equation}\label{e:Gdef}
	\chi(t):=\frac{t^{2-\ca} f_0(t)}{(1+f(t))^{2-\cc} f(t) g^{\frac{\cb}{A}}(t)}= \frac{  g^{-\frac{2\cb}{A}}(t) t^{2(1-\ca)}}{B f(t) (1+f(t))^{2(1-\cc)}}  .
\end{equation}
We note that the second equality holds by directly applying identity \eqref{e:f0frl} in Lemma \ref{t:f0fg}, and we recall that $B$ is defined there.

Before proceeding, we first present the following useful lemma.
\begin{lemma}\label{t:limgtf}
	Suppose $g(t)$ is defined by \eqref{e:gdef} and that $f\in C^2([t_0,t_m))$ (where $[t_0,t_m)$ is the maximal interval of existence of $f$) solves the ODE system \eqref{e:feq0}--\eqref{e:feq1}. Then, the quantity $g^{\frac{\cb}{A}}(t) t^{\ca-1} (1+f(t) )^{1-\cc} >0$ is bounded on $[t_0,t_m)$, and
	\begin{equation*}
		\lim_{t\rightarrow t_m} \bigl( g^{\frac{\cb}{A}}(t) t^{\ca-1} (1+f(t) )^{1-\cc} \bigr)= 0 .
	\end{equation*}
\end{lemma}

\begin{proof}
\underline{$(1)$ If $t_m<\infty$,} since  $g(t) \in(0,1]$ for $t\in[t_0,t_m)$ by Lemma \ref{t:f0fg}.$(3)$, we have
	\begin{equation*}
		0 \leq g^{\frac{\cb}{A}} t^{\ca-1} (1+f )^{1-\cc} \leq t_m^{\ca-1} (1+f )^{1-\cc}.
	\end{equation*}
By Theorem \ref{t:mainthm0}(2) and the fact that $\cc>1$,  it follows that $ \lim_{t\rightarrow t_m}    (1+f )^{1-\cc}=0$, which proves the lemma in this case.
	
	\underline{$(2)$ If $t_m=\infty$,} Using Proposition \ref{t:bksl} and the fact that  $g(t) \in(0,1]$ for $t\in[t_0,t_m)$, we obtain the estimate
	\begin{equation*}
		0\leq g^{\frac{\cb}{A}} t^{\ca-1} (1+f )^{1-\cc} \leq t^{\ca-1} \exp \bigl( (1-\cc)(\mathtt{C} t^{\frac{\ba+\triangle}{2} }  +\mathtt{D}  t^{-1})\bigr) .
	\end{equation*}
	Since, using the fact $\lim_{x\rightarrow \infty}(x^a/e^x)=0$ and noting $-\ba>0$, we obtain
	\begin{equation*}
		\lim_{t\rightarrow \infty}  \bigl[  t^{\ca-1} \exp \bigl( (1-\cc)(\mathtt{C} t^{\frac{\ba+\triangle}{2} }  +\mathtt{D}  t^{-1})\bigr) \bigr]=	\lim_{t\rightarrow \infty}  \bigl[  t^{-\ba} \exp \bigl( (1-\cc) \mathtt{C} t^{\frac{\ba+\triangle}{2} } \bigr) \bigr]=0.
	\end{equation*}

Since $g^{\frac{\cb}{A}} t^{\ca-1} (1+f )^{1-\cc} $  is continuous on $[t_0,t_m)$, by Proposition \ref{t:cont2} (see Appendix \ref{s:cont}), we conclude that it is bounded, which completes the proof of this lemma.
\end{proof}

\begin{lemma}\label{t:limtf}
	Suppose $f\in C^2([t_0,t_m))$ (where $[t_0,t_m)$ is the maximal interval of existence of $f$ given by Theorem \ref{t:mainthm0}) and that  $f$ solves the ODE system \eqref{e:feq0}--\eqref{e:feq1}. Assume that $t_m<\infty$ is finite. Then $
	\lim_{t\rightarrow t_m}[(t_m-t)f^{2-\cc}(t)]=+\infty$.
\end{lemma}
\begin{proof} 
Since $\lim_{t\rightarrow t_m }(t_m-t)=0$ and $\lim_{t\rightarrow t_m }f^{2-\cc}(t)=\infty$ (due to $\cc\in (1,3/2)$ and Theorem \ref{t:mainthm0}(5)), by applying L'Hospital rule and using the representation of $f_0$ from \eqref{e:f0frl}, we obtain the following through straightforward calculations
	\begin{align}\label{e:tfest}
		\lim_{t\rightarrow t_m}[(t_m-t)f^{2-\cc}(t)]\overset{\text{L'Hospital}}{=} & \frac{1}{2-\cc}	\lim_{t\rightarrow t_m} \frac{1}{f^{\cc-3}f_0} \overset{\eqref{e:f0frl}}{=} \frac{B t_m^{\ca}}{2-\cc}	\lim_{t\rightarrow t_m} \frac{f^{3-2\cc} }{ g^{-\frac{\cb}{A}}(t) } .
	\end{align}

Let us now analyze this limit in two cases, assuming $\lim_{t\rightarrow t_m}g (t)>0$ or $=0$. First, if $g(t_m):=\lim_{t\rightarrow t_m}g (t)>0$ (since $g(t)\in (0,1]$ for $t\in[t_0,t_m)$ by Lemma \ref{t:f0fg}.$(3)$), then by \eqref{e:tfest}, we have $\lim_{t\rightarrow t_m}[(t_m-t)^{2-\cc}f(t)]=+\infty$ by Theorem \ref{t:mainthm0}.$(5)$.
On the other hand, if $g(t_m):=\lim_{t\rightarrow t_m}g (t)=0$, applying L'Hospital rule again, we have  
	\begin{align*}\label{e:tfest2}
	\lim_{t\rightarrow t_m}[(t_m-t)f^{2-\cc}(t)]
	\overset{\eqref{e:f0frl}}{=} & \frac{B t_m^{\ca}}{2-\cc}	\lim_{t\rightarrow t_m} \frac{f^{3-2\cc} }{ g^{-\frac{\cb}{A}}(t) }
	\overset{\text{L'Hospital}}{=}  \frac{B (3-2\cc) t_m^{\ca}}{2-\cc}	\lim_{t\rightarrow t_m} \frac{f^{2-2\cc} f_0 }{ \del{t} g^{-\frac{\cb}{A}}(t) } \notag  \\	
	\overset{\eqref{e:dtgf}\&\eqref{e:f0frl}}{=} &  \frac{ B^{-1}(3-2\cc)  }{ \cb (2-\cc)t_m^{\ca-2}}	\lim_{t\rightarrow t_m}   g^{-\frac{\cb}{A}}(t)   =+\infty .
\end{align*}
The proof is complete.
\end{proof}

\begin{lemma}\label{t:gmap}
	Suppose $g(t)$ is defined by \eqref{e:gdef},   $\cc\in(1,3/2)$,  and  $f\in C^2([t_0,t_m))$ (where $[t_0,t_m)$ is the maximal interval of existence of $f$ given by Theorem \ref{t:mainthm0}), and suppose $f$ solves the  ODE system  \eqref{e:feq0}--\eqref{e:feq1}. Then, $
	\lim_{t\rightarrow t_m}g(t)=0$.
\end{lemma}
\begin{remark}
	Due to this lemma, it is convenient to continuously extend $g(t)$ from $[t_0,t_m)$ to $[t_0,t_m]$ by  defining  $g(t_m):=\lim_{t\rightarrow t_m}g(t)=0$, so that $g^{-1} (0)=t_m$.
\end{remark}

\begin{proof}
	By Lemma \ref{t:f0fg}.$(3)$, for $t\in[t_0,t_m)$, we have
	\begin{equation*}
		g^{-\frac{\cb}{A}} (t) = 1+ \cb B \int^t_{t_0} s^{\ca-2} f(s)(1+f(s))^{1-\cc}  ds  .
	\end{equation*}
Hence,  to conclude $
\lim_{t\rightarrow t_m}g(t)=0$, we only need to prove that $
\lim_{t\rightarrow t_m}g^{-\frac{\cb}{A}}(t)=\infty$. Furthermore, we need to verify $
\int^{t_m}_{t_0} s^{\ca-2} f(s)(1+f(s))^{1-\cc}  ds =\infty$.

First if $t_m<+\infty$, we aim to prove $
	\int^{t_m}_{t_0} s^{\ca-2} f(s)(1+f(s))^{1-\cc}  ds =\infty$,
that is, we will prove it is a divergence  improper integrals. Note the limit, with the help of Lemma \ref{t:limtf},
\begin{equation*}
\lim_{t\rightarrow t_m}	\frac{t^{\ca-2} f(t)(1+f(t))^{1-\cc}}{1/(t_m-t)} =t_m^{\ca-2}\lim_{t\rightarrow t_m} [(t_m-t) f(t)(1+f(t))^{1-\cc} ]=t_m^{\ca-2}\lim_{t\rightarrow t_m} [(t_m-t) f^{2-\cc}(t)]=+\infty,
\end{equation*}
and since the improper integral
\begin{equation*}
	\int^{t_m}_{t_0}\frac{1}{t_m-s} ds=+\infty,
\end{equation*}
that is, it is a divergent improper integral. By the comparison test for improper integrals, we obtain $
\int^{t_m}_{t_0} s^{\ca-2} f(s)(1+f(s))^{1-\cc}  ds =\infty$. Thus, by \eqref{e:gba}, we conclude that $\lim_{t\rightarrow t_m} g^{-\frac{\cb}{A}}(t)=\infty$.

On the other hand, if $t_m=+\infty$,
we aim to prove $
\int^{\infty}_{t_0} s^{\ca-2} f(s)(1+f(s))^{1-\cc}  ds =\infty$. If $\ca \geq 2$, then,  by Theorem \ref{t:mainthm0}.$(2)$, we have $
	\lim_{t\rightarrow +\infty} [ t^{\ca-2} f(t)(1+f(t))^{1-\cc}]    =\infty$. If $\ca \in(1, 2)$, then by \eqref{e:fbds} in Proposition \ref{t:bksl} and the fact that $\lim_{x\rightarrow \infty}(x^a/e^x)=0$, we obtain
	\begin{align*}
		&\lim_{t\rightarrow +\infty} \bigl[ t^{\ca-2} f(t)(1+f(t))^{1-\cc}  \bigr] =  \lim_{t\rightarrow +\infty} \bigl[ t^{\ca-2} (1+f(t))^{\frac{1}{2}} f^{\frac{3}{2}-\cc}(t) \bigr] \notag  \\
		&\hspace{1cm}>   \lim_{t\rightarrow +\infty} \Bigl[ t^{\ca-2}  \exp \Bigl(\frac{1}{2}( \mathtt{C} t^{\frac{\ba+\triangle}{2} }  +\mathtt{D}  t^{-1})\Bigr)   f^{\frac{3}{2}-\cc}(t) \Bigr]=\lim_{t\rightarrow +\infty} \Bigl[ t^{\ca-2}  \exp \Bigl(\frac{1}{2} \mathtt{C} t^{\frac{\ba+\triangle}{2} }   \Bigr)   f^{\frac{3}{2}-\cc}(t) \Bigr]=\infty.
	\end{align*}
That is, $\lim_{t\rightarrow +\infty} \bigl[ t^{\ca-2} f(t)(1+f(t))^{1-\cc}  \bigr] =\infty$ for any $\ca>1$. This leads to the conclusion that, since  $t^{\ca-2} f(t)(1+f(t))^{1-\cc}\in C^2([t_0,t_m))$ and it is strictly positive for all $t\in [t_0,t_m)$, by the property of continuous function, there is a constant $M>0$, such that $t^{\ca-2} f(t)(1+f(t))^{1-\cc} \geq M>0$ for all $t\in [t_0,\infty)$. Therefore,
\begin{equation*}
	\int^{\infty}_{t_0} s^{\ca-2} f(s)(1+f(s))^{1-\cc}  ds \geq M \int^{\infty}_{t_0}  ds=\infty ,
\end{equation*}
which, by \eqref{e:gba},  in turn implies $\lim_{t \rightarrow \infty} g(t) = 0$. This completes the proof.
\end{proof}

\begin{proposition}\label{t:limG}
	Suppose $\cc\in(1,3/2)$, $\cb>0$, $\ca>1$,  and $\chi$ is defined by \eqref{e:Gdef}. Let $f\in C^2([t_0,t_m))$, where $[t_0,t_m)$ is the maximal interval of existence of $f$ as given in Theorem \ref{t:mainthm0}, and assume that $f$ solves the ODE system \eqref{e:feq0}--\eqref{e:feq1}.
Then there exists a function $\mathfrak{G} \in C^1([t_0,t_m))$, such that for $t\in [t_0,t_m)$,
\begin{equation*}
	\chi(t)=\frac{2\cb B}{3-2\cc}+\mathfrak{G}(t).
\end{equation*}
where $\lim_{t\rightarrow t_m}\mathfrak{G}(t)=0$.
Moreover, there is a constant $C_\chi>0$ such that $0<\chi(t) \leq C_\chi$ for all $t\in [t_0,t_m)$. Additionally, there are continuous extensions of $\chi$ and $\mathfrak{G}$ such that $\chi\in C^0([t_0,t_m])$ and $\mathfrak{G}\in C^0([t_0,t_m])$, where we define  $\chi(t_m):=2\cb B/(3-2\cc)$ and $\mathfrak{G}(t_m):=0$.
\end{proposition}

\begin{proof}
Using \eqref{e:f0frl}, we first calculate
\begin{align}\label{e:long1}
	\del{t}(B t^{2(\ca-1)} f (1+f)^{2(1-\cc)})
	=&   (3-2\cc)   t^{\ca-2}  (1+f)^{2-\cc}   g^{-\frac{\cb}{A}}   - 2(1-\cc)   t^{\ca-2}     (1+f)^{1-\cc}     g^{-\frac{\cb}{A}}   \notag  \\
	&+2 B (\ca-1)t^{2\ca-3}f(1+f)^{2(1-\cc)}.
\end{align}

By Theorem \ref{t:mainthm0}, and since  $\cc\in(1,3/2)$ implies  $2(1-\cc)\in(-1,0)$ and $3-2\cc\in(0,1)$, we have the following limit,
	\begin{equation*}
		\lim_{t \rightarrow t_m}\bigl[f(1+f)^{2(1-\cc)}\bigr]=\lim_{t \rightarrow t_m}\bigl[(1+f)^{3-2\cc}-(1+f)^{2(1-\cc)}\bigr]=\infty.
	\end{equation*}
	With the help of this limit and the fact that $\lim_{t \rightarrow t_m}g^{-\frac{2\cb}{A}}(t)=\infty$ (due to Lemma \ref{t:gmap}), we  apply the L'Hospital rule using \eqref{e:f0frl}, \eqref{e:dtgf},  \eqref{e:Gdef} and Proposition \ref{t:limgtf},
\begin{align*}
	\lim_{t\rightarrow t_m} \chi(t) = & \lim_{t\rightarrow t_m} \frac{  g^{-\frac{2\cb}{A}} }{B t^{2(\ca-1)} f (1+f)^{2(1-\cc)}} \overset{\text{L'Hospital}}{=} \lim_{t\rightarrow t_m} \frac{  2g^{-\frac{\cb}{A}}\del{t} g^{-\frac{\cb}{A} } }{\del{t}(B t^{2(\ca-1)} f (1+f)^{2(1-\cc)})}  \notag  \\
	\overset{\eqref{e:dtgf}\&\eqref{e:long1}}{=} & \lim_{t\rightarrow t_m} \frac{ 2 \cb B  }{  (3-2\cc)    \frac{1+f}{f}      -   \frac{2(1-\cc)  }{f}         +2 B (\ca-1)t^{\ca-1}g^{\frac{\cb}{A}}  (1+f)^{1-\cc}}  \notag  \\
	\overset{\text{Lemma } \ref{t:limgtf}}{=}  &   \frac{  2 \cb B }{3-2\cc}   ,
\end{align*}
which means that  $\lim_{t\rightarrow t_m}\mathfrak{G}(t)=0$. Since $\chi(t)>0$ (which follows from its definition in \eqref{e:Gdef}) and is a continuous function on  $[t_0,t_m)$, we apply Proposition \ref{t:cont2} (see Appendix \ref{s:cont}) to conclude its boundedness, thus completing the proof. 
\end{proof}

The second crucial quantity in the Fuchsian formulation is
\begin{equation}\label{e:xidef}
	\xi(t):=1/[g(t) (1+f(t)) ]. 
\end{equation}
The next proposition proves that $\xi$ is bounded and that its limit approaches zero as $t$ tends to $t_m$. 

\begin{proposition}\label{t:fginv}
 		Suppose that $f\in C^2([t_0,t_m))$, where $[t_0,t_m)$ is the maximal interval of existence for $f$ given by Theorem \ref{t:mainthm0}, and that $f$ solves the ODE system \eqref{e:feq0}--\eqref{e:feq1}. Let $g(t)$ be defined by \eqref{e:gdef} and $\xi(t)$ be given by \eqref{e:xidef}. Then $\xi\in C^1([t_0,t_m))$ and
 		\begin{equation}\label{e:fginv}
 			\lim_{t\rightarrow t_m} \xi(t)= 0 .
 		\end{equation}
 		Moreover, there is a constant $C_\star>0$ such that $0<\xi(t) \leq C_\star$ for all $t\in[t_0,t_m)$. Additionally, there is a continuous extension of $\xi$ satisfying $\xi\in C^0([t_0,t_m])$, obtained by defining $\xi(t_m) := 0$.

\end{proposition}

\begin{proof}
	By the definition in \eqref{e:Gdef} of $\chi(t)$ and Proposition \ref{t:limG} (i.e., $0<\chi\leq C_\chi$), we obtain
	\begin{equation}\label{e:gest2}
		0< g^{-\frac{2\cb}{A}}(t)  \leq C_\chi B f(t) (1+f(t))^{2(1-\cc)} t^{2(\ca-1)}.
	\end{equation}
Since $1+f>f$, applying  \eqref{e:gest2}, we arrive at
\begin{align*}
	\xi^{\frac{2\cb}{A}} =g^{-\frac{2\cb}{A}}  (1+f)^{-\frac{2\cb}{A}} <   \frac{ (1+f )^{1-\frac{2\cb}{A}} }{g^{\frac{2\cb}{A}} f}\leq B C_\chi  t^{2(\ca-1)}(1+f)^{3-2\cc-\frac{2\cb}{A}} .
\end{align*}
Thus,
\begin{equation}\label{e:gfest}
	0<\xi <   B^{\frac{A}{2\cb}}  C_\chi^{\frac{A}{2\cb}} t^{A(\ca-1)/\cb}(1+f)^{\frac{A}{2\cb}(3-2\cc-\frac{2\cb}{A})}.
\end{equation}
Since $A\in(0,2\cb/(3-2\cc))$ (as recalled in \eqref{e:gdef}) and $\cc\in(1,3/2)$, it follows that  $3-2\cc-\frac{2\cb}{A}<0$.  Now, we consider two cases: $(1)$ If $t_m<\infty$, then, directly applying \eqref{e:gfest} and Theorem \ref{t:mainthm0} (i.e., $\lim_{t\rightarrow t_m} f(t)=+\infty$), we conclude \eqref{e:fginv}. $(2)$ If $t_m=\infty$, 
using the fact that $3-2\cc-\frac{2\cb}{A}<0$ and \eqref{e:fbds} in Proposition \ref{t:bksl}, the right-hand side of \eqref{e:gfest} can be estimated as follows 
\begin{equation*}
	t^{A(\ca-1)/\cb}(1+f)^{\frac{A}{2\cb}(3-2\cc-\frac{2\cb}{A})} < t^{A(\ca-1)/\cb} \exp\Bigl[\frac{A}{2\cb}\Bigl(3-2\cc-\frac{2\cb}{A}\Bigr) \bigl( \mathtt{C} t^{\frac{\ba+\triangle}{2} }  +\mathtt{D}  t^{-1}\bigr)\Bigr].
\end{equation*}
Then, using the fact that $\lim_{x\rightarrow \infty}(x^a/e^x)=0$,
\begin{equation*}
\lim_{t \rightarrow \infty} \Bigl(t^{A(\ca-1)/\cb} \exp\Bigl[\frac{A}{2\cb}\Bigl(3-2\cc-\frac{2\cb}{A}\Bigr) \bigl( \mathtt{C} t^{\frac{\ba+\triangle}{2} }  +\mathtt{D}  t^{-1}\bigr)\Bigr]\Bigr)=\lim_{t \rightarrow \infty} \biggl(\frac{t^{A(\ca-1)/\cb}}{\exp\bigl[-\frac{A}{2\cb}\bigl(3-2\cc-\frac{2\cb}{A}\bigr)  \mathtt{C} t^{\frac{\ba+\triangle}{2} }   \bigr]} \biggr)=0.
\end{equation*}
This means that the right-hand side of \eqref{e:gfest} tends to $0$ as $t \rightarrow t_m$, which implies \eqref{e:fginv}.
Since $\xi(t)$ is a continuous function on $[t_0,t_m)$, applying Proposition \ref{t:cont2} (see Appendix \ref{s:cont}), we conclude its boundedness, completing the proof. 
\end{proof}


\section{The analysis of the perturbed solutions}\label{s:fuchian}		
This section focuses on analyzing the main second order nonlinear hyperbolic equations \eqref{e:maineq0}--\eqref{e:maineq1} with parameters satisfying \eqref{e:abcdk}. For reference, these equations are given by
 \begin{gather}
	\Box \varrho(x^\mu) +\frac{\mathcal{a} }{t} \del{t}\varrho(x^\mu) -
	\frac{\mathcal{b}}{t^2} \varrho(x^\mu) (1+  \varrho(x^\mu) ) -\frac{\mathcal{c}-\ck}{1+\varrho(x^\mu)} (\del{t}\varrho(x^\mu))^2=  \ck F(t) , \label{e:maineq0a}   \\
	\varrho|_{t=t_0}= \mathring{\varrho}(x^i) \AND 	\del{t} \varrho|_{t=t_0}=   \mathring{\varrho}_0(x^i)  . \label{e:maineq1a}
\end{gather}

The \textit{goal} of this section is to prove Theorem \ref{t:mainthm1}. We recall it here for convenience. 
\begin{theorem}\label{t:fuc1}
	Suppose $s\in \Zbb_{\geq \frac{n}{2}+3}$, and let  $\ca, \cb, \cc, \ck$ be constants satisfying  \eqref{e:abcdk}.  Assume that $f\in C^2([t_0,t_m))$ is the solution to \eqref{e:feq0}--\eqref{e:feq1} given by Theorem \ref{t:mainthm0}, where $\mf>0$ and $\mf_0>0$ are prescribed. Furthermore, assume that $t_m>t_0$ so that $[t_0,t_m)$ is the maximal interval of existence for $f$ as stated in Theorem \ref{t:mainthm0}. 
	Then, there exist small constants  $\sigma_\star,\sigma>0$, such that if the initial data satisfies
	\begin{equation*} 
		\Bigl\|\frac{\mathring{\varrho}}{\mf}-1\Bigr\|_{H^s(\Tbb^n)}+  	\Bigl\|\frac{\mathring{\varrho}_0}{\mf_0}-1\Bigr\|_{H^s(\Tbb^n)}+  	\Bigl\|\frac{\cm  \mathring{\varrho}_i}{1+\mf}\Bigr\|_{H^s(\Tbb^n)} \leq \frac{1}{2} \sigma_\star\sigma ,
	\end{equation*}
	then there is a  solution $\varrho\in C^2([t_0,t_m)\times \Tbb^n)$ to the equation \eqref{e:maineq0}--\eqref{e:maineq1} satisfying the estimate
	\begin{equation}\label{e:mainest}
		\Bigl\|\frac{\varrho(t)}{f(t)}-1\Bigr\|_{H^s(\Tbb^n)}+  	\Bigl\|\frac{\del{t}\varrho(t)}{f_0(t)}-1\Bigr\|_{H^s(\Tbb^n)}+ 	\Bigl\|\frac{\cm \del{i}\varrho(t)}{1+f(t)}\Bigr\|_{H^s(\Tbb^n)}\leq   C \sigma<1
	\end{equation}
	for $t\in[t_0,t_m)$, where $C>0$ is some constant.
	Moreover, $\varrho$  blows up at $t=t_m$, i.e.,
	\begin{equation*}
		\lim_{t\rightarrow t_m} \varrho(t,x^i)=+\infty \AND \lim_{t\rightarrow t_m} \varrho_0(t,x^i)=+\infty ,
	\end{equation*}
	with the rate estimates $(1-C\sigma)f \leq  \varrho \leq (1+C\sigma)f$ and $(1-C\sigma)f_0 \leq \varrho_0 \leq (1+C\sigma)f_0$ 	for $t\in[t_0,t_m)$.
\end{theorem}

\begin{remark} 
	The regularity condition $s\in \mathbb{Z}_{\geq \frac{n}{2}+3}$ on the initial data in Theorem \ref{t:fuc1} may not be optimal, as the proof does not take into account the semilinearity of these equations \eqref{e:maineq0a}. As pointed out in \cite[Remark 4.3]{Beyer2020}, regularity improvements can be achieved by establishing global existence theorems for Fuchsian systems in their semilinear form (analogous to Appendix \ref{s:fuc}) which are adapted to semilinear equations.
\end{remark}

The basic idea behind proving this theorem is to transform \eqref{e:maineq0a}--\eqref{e:maineq1a} into a Fuchsian formulation (see the model \eqref{e:model1}--\eqref{e:model2} in Appendix \ref{s:fuc}). Then, roughly speaking, we use the results of this formulation (i.e., Theorem \ref{t:fuc} in Appendix \ref{s:fuc}) to derive the estimate \eqref{e:mainest}. We will elaborate this transformation in the next three steps (Step $(1)$--Step $(3)$) in the subsequent sections.

\subsection{Step $1$: Perturbation Equations}
Let $f(t)$ solve equations \eqref{e:feq0}--\eqref{e:feq1}. We then define
		\begin{align}
			w(t,x^i)& := \varrho(t,x^i)-f(t), \label{e:ww} \\
			w_{0}(t,x^i)&:=\partial_{t}w(t,x^i)=\partial_{t}\varrho(t,x^i)- f_0(t), \label{e:w0}
			\\
			w_{i}(t,x^i)&:=\partial_{i}w(t,x^i)=\partial_{i}\varrho(t,x^i) .\label{e:wi}
		\end{align}
Substituting the variables \eqref{e:ww}--\eqref{e:wi} into \eqref{e:maineq0a} and using equation \eqref{e:feq0} for $f$, we arrive at the following perturbation equation
\begin{align}\label{e:weq1}
	& \del{t} w_0 -  \mathcal{g}^{ij}\del{j} w_i+\frac{\ca}{t}  w_0 -\frac{\cb}{t^2}  (w+w^2+2fw) - \frac{\cc-\ck}{1+w+f} w^2_0  \notag  \\
	&\hspace{1cm} - \frac{2 (\cc-\ck) f_0 w_0}{1+w+f} + \frac{(\cc-\ck) f_0^2 w}{(1+w+f)(1+f)} =  0 .
\end{align}		
Through direct computation from \eqref{e:w0} and \eqref{e:wi}, we obtain
\begin{align}
	\del{t}w_i=&\del{i}\del{t}w=\del{i}w_0,  \label{e:weq2} \\
	\partial_{t}w =&w_{0} . \label{e:weq3}
\end{align}
Thus, equations \eqref{e:weq1}--\eqref{e:weq3} form a first order system.

\subsection{Step $2$: Time transformation and Fuchsian formulation}\label{s:stp2}

In this section, we aim to rewrite \eqref{e:weq1}--\eqref{e:weq3} into the Fuchsian formulation (see \eqref{e:model1} in Appendix \ref{s:fuc}). To achieve this, we first need to introduce an appropriate time transformation, as Fuchsian formulations rely\footnote{As noted in our previous works \cite{Liu2018b,Liu2018,Liu2022a,Liu2018a,Liu2022}, the key and most challenging steps in applying this approach involve constructing appropriate time compactifications and Fuchsian fields.} on a suitably \textit{compactified time coordinate}. The function $g(t)$ given in §\ref{t:ttf} serves as a viable candidate for this transformation.
Additionally, the Fuchsian formulation heavily depends on variable selection (i.e., \textit{Fuchsian fields}). A crucial requirement for obtaining the Fuchsian formulation is the choice of well-behaved, decaying Fuchsian fields.

To achieve these objectives, we introduce $(i)$ a \textit{time transformation},
\begin{align}\label{e:ttf}
	\tau := -g(t)=& -\exp\Bigl(-A\int^t_{t_0} \frac{f(s)(f(s)+1)}{s^2f_0(s)} ds\Bigr)  \notag  \\
	= & -\Bigl(1+ \cb B \int^t_{t_0} s^{\ca-2} f(s)(1+f(s))^{1-\cc}  ds \Bigr)^{-\frac{A}{\cb}}\in[-1,0),
\end{align}
where, as recalled in §\ref{t:ttf},  $g(t)$  is defined by \eqref{e:gdef} (or equivalently, by \eqref{e:gdef2}), and $A\in(0,2\cb/(3-2\cc))$ is an arbitrary constant.  This time transformation maps the initial time $t=t_0$ to $\tau=-1$ and the maximal time of existence $t=t_m$ to $\tau=0$.
In addition, we  introduce  $(ii)$ \textit{rescaled Fuchsian fields} as follows (recall that $\cm$ is a constant given in \eqref{e:Fdef}),
\begin{align}
    u(t,x^i)= &  \frac{1}{f(t)} w (t,x^i)	,  \label{e:u} \\
	u_0(t,x^i)= &  \frac{1}{f_0(t)}  w_0(t,x^i),
	\label{e:u0}\\
    u_i(t,x^i)=& \frac{\cm}{1+f(t)} w_i(t,x^i).  \label{e:ui}
\end{align}

Following the notation conventions in \eqref{e:udl} from \S\ref{s:AIN}, $\uu$  represents $u$ in the compactified time coordinate $\tau$. We will frequently use this notation in the following. For instance, we define:
\begin{equation}\label{e:udvar}
	\underline{u}(\tau,x^i)=u(g^{-1}(-\tau),x^i), \quad \uuo(\tau,x^i)=  	u_0(g^{-1}(-\tau),x^i) \AND \uui(\tau,x^i)=  u_i(g^{-1}(-\tau),x^i) .
\end{equation}

Under the time transformation \eqref{e:ttf}, using \eqref{e:dtg0}, we obtain the following expression for later use:\footnote{Note that the underlines follow the notation convention in \eqref{e:udl} from §\ref{s:AIN}.}
\begin{equation}\label{e:ttf2}
	\del{\tau} \underline{ u}_\mu= -\underline{[g^\prime(t)]^{-1} \del{t} u_\mu}
	= \underline{ \frac{t^{2-\ca} (1+f )^{\cc-1} }{A  B   g^{\frac{\cb}{A}+1}  f} \del{t} u_\mu } .
\end{equation}

\underline{(1) Rewrite \eqref{e:weq1}:} By using \eqref{e:u}--\eqref{e:ui}, we obtain
\begin{align}\label{e:var}
	w=   fu, \quad w_0=   f_0 u_0, \quad w_i=  \cm^{-1} (1+f)u_i.
\end{align}
Substituting these expressions \eqref{e:var} into \eqref{e:weq1}, we derive
\begin{align*}
	&    u_0 \del{t} f_0 +  f_0 \del{t} u_0 - \cm^{-1} (1+f)  \mathcal{g}^{ij} \del{j} u_i+\frac{\ca}{t}     f_0 u_0   -\frac{\cb}{t^2}  \bigl(   f u+(   f u)^2+2   f^2 u\bigr)  \notag  \\
	& \hspace{1cm} - \frac{(\cc-\ck)(   f_0   u_0)^2}{1+  f u+f} - \frac{2(\cc-\ck)     f_0^2  u_0}{1+   f u+f}  + \frac{(\cc-\ck)    u f  f_0^2}{(1+  f u+f)(1+f)} =0 .
\end{align*}
Using the ODE \eqref{e:feq0} for $f$ to substitute $\del{t}f_0$ in the first term above, and noting the cancellation of the term  $  \ca t^{-1} f_0 u_0$, we obtain
\begin{align}\label{e:stp1}
	&  \frac{  \cb}{t^2} f (1+f ) u_0+ \frac{ \cc f_0^2 u_0 }{1+f }  +   f_0 \del{t} u_0 -   \cm^{-1} (1+f)  \mathcal{g}^{ij}\del{j} u_i  -\frac{\cb}{t^2}  (  f u+  f^2 u^2+2   f^2 u) \notag  \\
	& \hspace{1cm} - \frac{(\cc-\ck)( f_0   u_0)^2}{1+  f u+f} - \frac{2(\cc-\ck)    f_0^2   u_0  }{1+  f u+f} + \frac{(\cc-\ck)   u f f_0^2}{(1+  f u+f)(1+f)} = 0.
\end{align}		
Multiplying both sides of \eqref{e:stp1} by $  1/(f_0  g^{\frac{\cb}{A}+1})$, we obtain
\begin{align}\label{e:stp3}
	&   - g^{-\frac{\cb}{A}-1}  \del{t} u_0 +  \frac{(1+f)\mathcal{g}^{ij} }{\cm  f_0  g^{\frac{\cb}{A}+1}} \del{j} u_i= \frac{  \cb f (1+f )}{t^2f_0g^{\frac{\cb}{A}+1}}   u_0+ \frac{ \cc f_0   g^{-\frac{\cb}{A}-1} u_0 }{1+f }  -\frac{  \cb f  }{t^2f_0g^{\frac{\cb}{A}+1} }   (  u+   f  u^2+2   f  u)  \notag  \\
	& \hspace{1cm}  - \frac{(\cc-\ck) f_0   g^{-\frac{\cb}{A}-1} u_0^2}{1+  f u+f}   - \frac{2(\cc-\ck)   f_0  g^{-\frac{\cb}{A}-1} u_0  }{1+  f u+f} + \frac{ (\cc-\ck)   u f f_0  g^{-\frac{\cb}{A}-1}}{(1+  f u+f)(1+f)} .
\end{align}	
Using the definition \eqref{e:Gdef} of $\chi$ and the expression \eqref{e:f0frl} for $f_0$,  we obtain
\begin{equation}\label{e:f0g}
	f_0 g^{-\frac{\cb}{A}} =  (1+f )^{2-\cc} f t^{\ca-2} \chi \AND 1/(f_0 g^{\frac{\cb}{A}})= B  t^{\ca} (1+f )^{-\cc}.
\end{equation} 
Multiplying $t^{2-\ca} (1+f )^{\cc-1}/(ABf)$ on both sides of \eqref{e:stp3} and using \eqref{e:f0g} to replace the corresponding terms in this equation \eqref{e:stp3}, we obtain 
\begin{align}\label{e:stp4}
	& - \frac{t^{2-\ca} (1+f )^{\cc-1}}{ A  B   g^{\frac{\cb}{A}+1}  f } \del{t} u_0  +   \frac{ t^{2 }   \mathcal{g}^{ij}}{A  \cm  f  g } \del{j} u_i
	=  \frac{  \cb}{A g}    u_0+ \frac{ \cc       \chi u_0 }{AB g} -\frac{  \cb}{  Ag(1+f )}   u-\frac{   \cb f }{  Ag(1+f )}     u^2\notag  \\
	& \hspace{1cm} -\frac{2   \cb f }{  Ag(1+f )}   u   - \frac{(\cc-\ck) (1+f)   \chi u_0^2}{AB(1+f+  f u)g} - \frac{2(\cc-\ck)    (1+f )    \chi u_0  }{AB(1+f+f u)g} + \frac{ (\cc-\ck)   u f   \chi}{AB(1+f+  f u) g}  .
\end{align}	
Then we note the following simple identities from straightforward calculations,
\begin{gather*}
	\frac{ f }{ 1+f }  =1-\frac{1 }{ 1+f }, \quad \frac{ 1+f   }{(1+f+  f u)g}=\frac{1}{g}-\frac{\mathfrak{R}(t,u)}{g} , \\
	\frac{ f   }{(1+f+  f u)g}=\frac{1}{g}-\frac{1}{(1+f)g}-\frac{\mathfrak{R}(t,u)}{ g}+\frac{\mathfrak{R}(t,u)}{ (1+f)g},
\end{gather*}
where
\begin{equation*}
	C^0\bigl([t_0,t_m], C^\infty(B_R(0))\bigr)\ni\mathfrak{R}(t,u):=\frac{  u}{1+1/f+  u}
\end{equation*}
for some constant $R>0$,  with $\mathfrak{R}(t, 0)=0$.
Using these identities, along with the definition \eqref{e:xidef} of $\xi$, we rewrite equation \eqref{e:stp4} as
\begin{align}\label{e:stp5}
& -  \frac{t^{2-\ca}    (1+f )^{\cc-1}}{A   B g^{\frac{\cb}{A}+1} f}  \del{t} u_0  +   \frac{ t^{2 }  \mathcal{g}^{ij} }{A   \cm   f  g}\del{j} u_i
=  \frac{1}{A g} \Bigl( \cb + \frac{(2 \ck-\cc)     \chi  }{ B } \Bigr) u_0 - \frac{(\cc-\ck) \chi u_0^2}{AB g}    \notag  \\
&\hspace{1cm} + \frac{(\cc-\ck)   \chi  u_0^2\mathfrak{R}(t,u) }{AB g} + \frac{2(\cc-\ck)      \chi u_0 	\mathfrak{R}(t,u) }{ABg}    + \frac{ 1}{A  g} \Bigl(\frac{(\cc-\ck)   \chi}{B} -2  \cb   \Bigr)  u -\frac{  \cb }{  Ag }     u^2 \notag  \\
&\hspace{1cm} -\frac{ (\cc-\ck)   u   \chi\mathfrak{R}(t,u)}{AB g}+\frac{  \cb \xi}{  A}     u^2 +\frac{  \cb \xi}{  A}   u   - \frac{ (\cc-\ck)   u   \chi \xi}{AB}  + \frac{ (\cc-\ck)   u  \chi \xi \mathfrak{R}(t,u) }{AB} .
\end{align}	
Using the definition \eqref{e:Fdef} of $\cg^{ij}$, the expression \eqref{e:f0frl} of $f_0$, and the definition \eqref{e:Gdef} of $\chi$, we obtain
\begin{equation}\label{e:gij}
	\cg^{ij}=\frac{\mathcal{m}^2 f_0^2}{(1+f)^2}\delta^{ij}  =  \mathcal{m}^2 B^{-2} t^{-2\ca} g^{-\frac{2\cb}{A}}(1+f)^{2\cc-2} \delta^{ij} =\mathcal{m}^2 B^{-1} f t^{-2} \chi \delta^{ij} .
\end{equation}
Then, with the help of Proposition \ref{t:limG}, by replacing $\chi$ with $\frac{2\cb B}{3-2\cc}+\mathfrak{G}(t)$ and using the expression \eqref{e:gij} of $\cg^{ij}$, we transform equation \eqref{e:stp5} into
\begin{align}\label{e:stp6}
	& -  \frac{t^{2-\ca} (1+f )^{\cc-1}}{ABg^{\frac{\cb}{A}+1} f}     \del{t} u_0  + \frac{\mathcal{m} \chi  }{AB g}  \delta^{ij} \del{j} u_i  \notag  \\
	= & \frac{1}{A g} \Bigl( \cb +   \frac{2\cb (2 \ck-\cc)    }{3-2\cc}+ B^{-1}(2 \ck-\cc)  \mathfrak{G}(t) +\mathfrak{H}(t,u_0,u)  \Bigr) u_0   \notag  \\
	&   + \frac{ 1}{A  g} \Bigl( \frac{2\cb (\cc-\ck)   }{3-2\cc}  -2  \cb +B^{-1}(\cc-\ck) \mathfrak{G}(t)  +	\mathfrak{F}(t, u) \Bigr) u   + \mathfrak{L}(t,u) 	,
\end{align}
where we define
\begin{align}
	\mathfrak{H}(t,u_0,u):= & - \frac{(\cc-\ck)\chi }{B} \bigl(   u_0 +     u_0 \mathfrak{R}(t,u) +  2   	\mathfrak{R}(t,u) \bigr) , \label{e:rm1} \\
	\mathfrak{F}(t,u):= & -  \cb u +  (\cc-\ck) B^{-1}     \chi\mathfrak{R}(t,u) ,\label{e:rm2} \\
	\mathfrak{L}(t,u):=&\frac{ \cb \xi}{  A}     u^2 +\frac{  \cb \xi}{  A}   u   - \frac{ (\cc-\ck)   u   \chi \xi}{AB}  + \frac{ (\cc-\ck)  u  \chi \xi \mathfrak{R}(t,u) }{AB} .  \label{e:rm3}
\end{align}
We can verify,  using Propositions \ref{t:limG} and \ref{t:fginv} for $\chi(t)$ and $\xi(t)$, that $\mathfrak{H}\in C^0\bigl([t_0,t_m], C^\infty(B_R(0)\times B_R(0))\bigr)$ and $\mathfrak{F}, \mathfrak{L}\in C^0\bigl([t_0,t_m], C^\infty(B_R(0))\bigr)$, with $\mathfrak{H}(t,0,0)=0$, $\mathfrak{F}(t,0)=0$, and  $\mathfrak{L}(t,0)=0$.

Now, let us transform the above equation \eqref{e:stp6}, which is expressed in terms of the independent variable $t$, into an equation where the independent variable is $\tau$, via the change of variable $t=g^{-1}(-\tau)$. Using the notation given in \eqref{e:udvar} and the expression for $\del{\tau}\uuo$ from \eqref{e:ttf2}, we arrive at 
\begin{align}\label{e:fuc1}
	 \del{\tau} \underline{ u}_0 +         \frac{\cm \underline{\chi}}{AB\tau} \delta^{ij} \del{j}  \uui
	=  &  \frac{1}{A \tau} \Bigl( \cb +   \frac{2\cb (2 \ck-\cc)    }{3-2\cc}+ B^{-1}(2 \ck-\cc)  \underline{\mathfrak{G}}(\tau)   +\underline{\mathfrak{H}}(\tau,\uuo,\uu)  \Bigr) \uuo   \notag  \\
	&   + \frac{ 1}{A  \tau} \Bigl( \frac{2\cb (\cc-\ck)   }{3-2\cc}  -2  \cb +B^{-1}(\cc-\ck) \underline{\mathfrak{G}}(\tau)  +	\underline{\mathfrak{F}}(\tau, \uu) \Bigr) \uu   - \underline{\mathfrak{L}}(\tau, \uu) 	,
\end{align}
where, by \eqref{e:rm1}--\eqref{e:rm3} and the notation conventions introduced in \S\ref{iandc},  $\underline{\mathfrak{G}}(\tau) := \mathfrak{G}(g^{-1}(-\tau))$, $\underline{\mathfrak{R}}(\tau,\uu)= \mathfrak{R}(g^{-1}(-\tau),u(g^{-1}(-\tau),x^i))$ and
\begin{align}
	\underline{\mathfrak{H}}(\tau,\uuo,\uu):= & - \frac{(\cc-\ck)\underline{\chi} }{B} \bigl(   \uuo +     \uuo \underline{\mathfrak{R}}(\tau,\uu) +  2   	\underline{\mathfrak{R}}(\tau,\uu) \bigr) , \label{e:Rmh}\\
	\underline{\mathfrak{F}}(\tau,\uu):= & -  \cb \uu +  (\cc-\ck) B^{-1}     \underline{\chi}\underline{\mathfrak{R}}(\tau,\uu) , \label{e:Rmf}\\
	\underline{\mathfrak{L}}(\tau,\uu):=&\frac{ \cb \underline{\xi}}{  A}     \uu^2 +\frac{  \cb \underline{\xi}}{  A}   \uu   - \frac{ (\cc-\ck)   \uu   \underline{\chi} \underline{\xi}}{AB}  + \frac{ (\cc-\ck)  \uu  \underline{\chi} \underline{\xi} \underline{\mathfrak{R}}(\tau,\uu) }{AB} .  \label{e:Rml}
\end{align}
We can verify, once again, using Propositions \ref{t:limG} and \ref{t:fginv}, that since $\xi(t)$ and $\chi(t)$ are continuous and bounded, we have the following regularity properties:  $\underline{\mathfrak{H}}\in C^0\bigl([-1,0], C^\infty(B_R(0)\times B_R(0))\bigr)$, and $\underline{\mathfrak{F}}, \underline{\mathfrak{L}} \in C^0\bigl([-1,0], C^\infty(B_R(0))\bigr)$, and $\underline{\mathfrak{H}}(\tau,0,0)=0$, $\underline{\mathfrak{F}}(\tau,0)=0$, $\underline{\mathfrak{L}}(\tau,0)=0$.

\underline{$(2)$ Rewrite \eqref{e:weq2}:}			
Inserting \eqref{e:var} into \eqref{e:weq2}, we obtain
\begin{equation*}
	\del{t} u_i -  \frac{\cm  f_0}{1+f}  \del{i}u_0 = -  \frac{ f_0}{1+f}  u_i .
\end{equation*}	
Multiplying this equation by $- 1/(A B t^{\ca-2}g^{\frac{\cb}{A}+1} f (1+f)^{1-\cc} )$ and using \eqref{e:f0frl} to substitute $f_0$, we get
\begin{equation*}
	-  \frac{ t^{2-\ca} (1+f)^{\cc-1}  }{	A  B g^{\frac{\cb}{A}+1}    f  }	 \del{t} u_i  +\frac{  \cm t^{2-2 \ca} (1+f)^{2\cc-2} }{	A  B^2 g^{\frac{2 \cb}{A}+1}   f   }   \del{i}u_0 = \frac{  t^{2-2\ca} (1+f)^{2\cc-2} }{	A  B^2 g^{\frac{2\cb}{A}+1}     f  } u_i  .
\end{equation*}	
Applying the definition \eqref{e:Gdef} of $\chi$, this equation simplifies to
\begin{equation*}
	-  \frac{t^{2-\ca} (1+f)^{\cc-1} }{	A  B g^{\frac{\cb}{A}+1}    f   }	\del{t} u_i   +\frac{  \cm \chi }{	A  B   g    } \del{i}u_0 = \frac{ \chi }{	A  B g   } u_i .
\end{equation*}	
Then, using \eqref{e:ttf2}, we transform to the independent variable $\tau$ via the substitution $t=g^{-1}(-\tau)$. Expressing the equation in terms of the variables defined in \eqref{e:udvar}, applying the definition \eqref{e:Gdef} of $\chi$, and multiplying both sides by $\delta^{ij}$, we arrive at
\begin{equation}\label{e:fuc2}
  \delta^{ij}\del{\tau} \uui  +\frac{\cm \underline{\chi}  }{AB\tau}\delta^{ij}  \del{i} \uuo= \frac{1}{\tau}\frac{ \underline{\chi}      }{	A  B    }	\delta^{ij}   	\uui.
\end{equation}

\underline{$(3)$ Rewrite \eqref{e:weq3}:}		
Inserting \eqref{e:var} into \eqref{e:weq3}, a straightforward calculation gives
\begin{align*}
	\partial_{t}u =   \frac{f_0}{f}   u_0- \frac{f_0}{f}   u .
\end{align*}
Multiplying this equation by $- 1/(A B t^{\ca-2}g^{\frac{\cb}{A}+1} f (1+f)^{1-\cc} )$ and using the definition \eqref{e:Gdef} of $\chi$ to substitute the corresponding terms, we obtain
\begin{equation}\label{e:s3stp1}
	-\frac{t^{2-\ca}(1+f)^{\cc-1}}{A  B g^{\frac{\cb}{A}+1}    f   } \del{t}u =\frac{1}{g}\Bigl(- \frac{1 }{A  B } \frac{f+1}{f} \chi u_0+\frac{1}{A  B }  \frac{f+1}{f} \chi    u \Bigr).
\end{equation}
Then, according to the definition \eqref{e:xidef} of $\xi$, we obtain
\begin{equation}\label{e:invfg}
	\frac{1}{fg}
	=\xi\Bigl(1+\frac{1}{f}\Bigr).
\end{equation}
With the help of \eqref{e:invfg} and Proposition \ref{t:limG}, substituting into \eqref{e:s3stp1} yields
\begin{equation*}
	-\frac{ t^{2-\ca} (1+f)^{\cc-1} }{A  B g^{\frac{\cb}{A}+1}   f   } \del{t}u=\frac{1}{A g}\Bigl(  \Bigl(- \frac{2\cb  }{3-2\cc}-\frac{1 }{ B }\mathfrak{G}(t)\Bigr)   u_0 +  \Bigl( \frac{2\cb  }{3-2\cc}+\frac{1}{  B }\mathfrak{G}(t)\Bigr)    u\Bigr) + \mathfrak{K}(t,u_0,u)
\end{equation*}
where
\begin{equation*}
	\mathfrak{K}(t,u_0,u):= - \frac{1 }{A  B } \chi \xi\Bigl(1+\frac{1}{f}\Bigr)u_0 + \frac{1}{A  B } \chi \xi\Bigl(1+\frac{1}{f}\Bigr) u .
\end{equation*}
Then, using \eqref{e:ttf2}, we transform to the independent variable $\tau$ via the substitution $t=g^{-1}(-\tau)$. Expressing the equation in terms of the variables defined in \eqref{e:udvar}, we arrive at
\begin{equation}\label{e:fuc3}
 \del{\tau}\uu  =\frac{1}{A \tau}\Bigl(  \Bigl(- \frac{2\cb  }{3-2\cc}-\frac{1 }{ B }\underline{\mathfrak{G}}(\tau)\Bigr)   \uuo +  \Bigl( \frac{2\cb  }{3-2\cc}+\frac{1}{  B }\underline{\mathfrak{G}}(\tau)\Bigr)    \uu \Bigr) - \underline{\mathfrak{K}}(\tau,\uuo,\uu)
\end{equation}
where
\begin{equation}\label{e:Rmk}
	\underline{\mathfrak{K}}(\tau,\uuo,\uu):= - \frac{1 }{A  B } \underline{\chi}\underline{\xi} \Bigl(1+  \frac{1}{ \underline{f}} \Bigr)  \uuo+ \frac{1}{A  B } \underline{\chi} \underline{\xi} \Bigl(1+  \frac{1}{ \underline{f}} \Bigr)  \uu .
\end{equation}
We can also verify, again by using Propositions \ref{t:limG} and \ref{t:fginv} for $\xi(t)$ and $\chi(t)$, that $\underline{\mathfrak{K}}\in C^0\bigl([-1,0], C^\infty(B_R(0)\times B_R(0))\bigr)$.

\subsection{Step $3$: Fuchsian formulations}
Gathering \eqref{e:fuc1}, \eqref{e:fuc2}, and \eqref{e:fuc3} together, we obtain the following system, 
\begin{equation}\label{e:fuc}
	\B^0\del{\tau}\U+\B^j\del{j}\U=\frac{1}{\tau}\mathfrak{B}\Pbb \U+ \mathcal{H}
\end{equation}
where $\U:=(\uuo, \uui, \uu)^T$, $
\mathcal{H}:=\mathcal{H}(\tau,\uuo,\uu)= \p{-\underline{\mathfrak{L}}(\tau, \uu), 0, -\underline{\mathfrak{K}}(\tau,\uuo,\uu)}^T$, $\Pbb:= \mathds{1}  $,
The coefficient matrices are given by
\begin{gather*}
	\B^0:=\p{1 & 0 & 0 \\
		0 &  \delta^{ki} & 0 \\
		0 & 0 & 1 }, \quad
	\B^j:= \frac{\cm \underline{\chi}}{AB\tau} \p{0 & \delta^{ij}  & 0 \\
	 \delta^{kj}  & 0 & 0 \\
	0 & 0 & 0} ,  \\
\mathfrak{B}:= \frac{1}{A}   \p{\cb +   (2 \ck-\cc)  (\frac{2\cb   }{3-2\cc}+   \frac{\underline{\mathfrak{G}} }{  B } )  +\underline{\mathfrak{H}}   & 0 & -2  \cb+(\cc-\ck) ( \frac{2\cb   }{3-2\cc}   +   \frac{\underline{\mathfrak{G}} }{  B })   +	\underline{\mathfrak{F}}  \\
	0 &   (\frac{2\cb }{3-2\cc}+\frac{\underline{\mathfrak{G}}}{B})	\delta^{ki}   & 0 \\
	- (\frac{2\cb  }{3-2\cc}+\frac{  \underline{\mathfrak{G}}  }{ B }) & 0 & \frac{2\cb  }{3-2\cc}+\frac{\underline{\mathfrak{G}} }{  B } },
\end{gather*}
The functions $\underline{\mathfrak{G}} :=\underline{\mathfrak{G}}(\tau)$, $\underline{\mathfrak{F}}: =\underline{\mathfrak{F}}(\tau, \uu) $, $\underline{\mathfrak{H}}  :=\underline{\mathfrak{H}}(\tau,\uuo,\uu) $,  $\underline{\mathfrak{L}}:=\underline{\mathfrak{L}}(\tau, \uu)$ and $\underline{\mathfrak{K}}:=\underline{\mathfrak{K}}(\tau,\uuo,\uu)$ are defined by \eqref{e:Rmh}--\eqref{e:Rml} and \eqref{e:Rmk}.

\subsection{Step $4$: Verification of the Fuchsian system}
In this section (Step 4), let us first verify that \eqref{e:fuc} satisfies all the conditions (i.e., Conditions \eqref{c:2}--\eqref{c:7} in Appendix \ref{s:fuc}) required for a Fuchsian system of the form \eqref{e:model1} when $\tau>\tau_\delta$ is close to $0$. Once these conditions are verified, we can then apply Theorem \ref{t:fuc} to conclude that a solution exists for $\tau\in[\tau_\delta,0)$ (Step 5). 
Then, by choosing sufficiently small initial data and applying the continuation principle, we extend the local solution beyond $\tau_\delta$. Combining this with the previously obtained solution near $\tau=0$, we thus obtain a global solution (Step 6). 
Finally, by transforming the Fuchsian fields back to the original variables $\varrho$ and $\del{t}\varrho$ (Step 7), we conclude Theorem \ref{t:mainthm1}.

In \eqref{e:fuc}, corresponding to the system \eqref{e:model1}--\eqref{e:model2}, we have
$\mathbf{P}=\Pbb=\mathds{1}$, $ B^{0} =\B^0=\mathds{1}$, and $B_2^i(\tau)=\tau\B^i(\tau)$ is independent of $(x^i, \U)$, while $B_0^i=0$. Moreover, we have
\begin{align*}
	\tilde{\mathbf{B}} =\tilde{\mathfrak{B}}= &  \frac{1}{A}   \p{\cb +   (2 \ck-\cc)  (\frac{2\cb   }{3-2\cc}+   \frac{\underline{\mathfrak{G}} }{  B } )     & 0 & -2  \cb+(\cc-\ck) ( \frac{2\cb   }{3-2\cc}   +   \frac{\underline{\mathfrak{G}} }{  B })    \\
			0 &   (\frac{2\cb }{3-2\cc}+\frac{\underline{\mathfrak{G}}}{B})	\delta^{ki}   & 0 \\
			- (\frac{2\cb  }{3-2\cc}+\frac{  \underline{\mathfrak{G}}  }{ B }) & 0 & \frac{2\cb  }{3-2\cc}+\frac{\underline{\mathfrak{G}} }{  B } } \notag  \\
		 \in & C^0([T_0,0],  \mathbb M_{(n+2)\times (n+2)}), 
\end{align*}
such that
\begin{equation*}
	\mathbf{B}(\tau, \U)-\tilde{\mathbf{B}} (\tau) =\mathfrak{B}(\tau, \U) -\tilde{\mathfrak{B}}(\tau)= \mathrm{O}(\U)
\end{equation*}
for all $ (t, u) \in[T_{0},0] \times B_R(\Rbb^{n+2}) $. 
In this case, we find that the parameter $\mathtt{b}$ defined in Theorem \ref{t:fuc} satisfies $\mathtt{b}=0$ (since both $\tilde{\mathbf{B}}$ and $\tilde{B}^i_2$ depend only on $\tau$).
Using this setting, there exists a constant $R>0$ (shrinking it if necessary), such that it follows immediately that conditions \eqref{c:2}--\eqref{c:4} and \eqref{c:6} are satisfied, given that $\mathbf{P}^\perp=\Pbb^\perp=0$ and $B^i_2=\tilde{B}^i_2=\tau\B^i(\tau)$. 
Since $B^0\equiv \mathds{1}$ is independent of $u$, we have $\Theta(s)\equiv B^0\equiv \mathds{1}$ and further $\Theta^\prime(0)\equiv 0$, which verifies \eqref{c:7}. Moreover, from this, we conclude that $\beta_{2\ell+1}=0$ for all $\ell=0,1,2,3$. 
Furthermore, we obtain $\gamma_1 \max\{\sum^3_{\ell=0} \beta_{2\ell+1}, \beta_1+2k(k+1)  \mathtt{b} \}=0$ which implies that we only need $\kappa>0$ to ensure that
$\kappa > (1/2) \gamma_1 \max \{\sum^3_{\ell=0} \beta_{2\ell+1}, \beta_1+2k(k+1)  \mathtt{b}  \}$.

In the following, we verify Condition \eqref{c:5}. To do so, we must show that there exist constants $\kappa$, $\gamma_{2}$, and $\gamma_{1}=1$ satisfying
\begin{equation*}
	 \mathds{1}= \B^{0}\leq \frac{1}{\kappa} \mathfrak{B} \leq\gamma_{2}\mathds{1}  \quad \text{i.e., }\quad \zeta^T \mathds{1} \zeta \leq \frac{1}{\kappa} \zeta^T\mathfrak{B}\zeta \leq\gamma_{2}\zeta^T\mathds{1} \zeta
\end{equation*}
for all $ (\tau,\U) \in[-1,0] \times B_R(\Rbb^{n+2}) $ and $\zeta \in \Rbb^{n+2}$. However, as we will see below, this condition cannot be fully satisfied. Nevertheless, we can still prove it within a much smaller region of $ (\tau,\U)$.

First, note that there exists a symmetric matrix $\mathds{B}$ defined by
\begin{align*}
\mathds{B}:= \frac{1}{A}   \p{\cb +   \frac{2\cb  (2 \ck-\cc)   }{3-2\cc}+   \frac{ (2 \ck-\cc) \underline{\mathfrak{G}} }{  B }    +\underline{\mathfrak{H}}   & 0 & - \cb+   \frac{ (\cc-\ck-1) \cb   }{3-2\cc}   +  \frac{(\cc-\ck-1)}{2} \frac{\underline{\mathfrak{G}} }{  B }    +	\frac{1}{2} \underline{\mathfrak{F}}  \\
	0 &   (\frac{2\cb }{3-2\cc}+\frac{\underline{\mathfrak{G}}}{B})	\delta^{ki}   & 0 \\
 - \cb+   \frac{ (\cc-\ck-1) \cb   }{3-2\cc}   +  \frac{(\cc-\ck-1)}{2} \frac{\underline{\mathfrak{G}} }{  B }    +	\frac{1}{2} \underline{\mathfrak{F}}   & 0 & \frac{2\cb  }{3-2\cc}+\frac{\underline{\mathfrak{G}} }{  B } },	
\end{align*}
such that
\begin{equation*}
 \zeta^T\mathfrak{B}\zeta= \zeta^T\mathds{B}\zeta .
\end{equation*}

Since $\mathds{B}$ is a real symmetric matrix, there exists an orthogonal matrix $\Omega$ such that
\begin{align*}
 \mathds{B}
 = \Omega^T \Lambda \Omega    \quad \text{where} \quad \Lambda:= \frac{1}{A} \p{\lambda_1 & 0 & 0 \\
	0 & \lambda_2 \delta^{ij}  & 0 \\
	0 & 0 & \lambda_3}. 
\end{align*}
The eigenvalues are given by
\begin{gather*}
	\lambda_1:= \tilde{ \lambda}_1+\mathfrak{Z}_1 (\underline{\mathfrak{G}} ,\underline{\mathfrak{H}},\underline{\mathfrak{F}})= \frac{\cb (4 c-4 \ck- 5 - \widetilde{\triangle} )}{2 (2 \cc-3)}+\mathfrak{Z}_1(\underline{\mathfrak{G}} ,\underline{\mathfrak{H}},\underline{\mathfrak{F}}),\notag \\
	\lambda_2:= \tilde{ \lambda}_2+\mathfrak{Z}_2 (\underline{\mathfrak{G}} , \underline{\mathfrak{H}},\underline{\mathfrak{F}})=   \frac{b (4 c-4 k-5+\widetilde{\triangle})}{2 (2 c-3)}+\mathfrak{Z}_2(\underline{\mathfrak{G}} ,\underline{\mathfrak{H}},\underline{\mathfrak{F}}),  \\
	\lambda_3 :=	\tilde{ \lambda}_3+\frac{\underline{\mathfrak{G}}}{B}=  \frac{2\cb }{3-2\cc}+\frac{\underline{\mathfrak{G}}}{B}     , \quad\text{where} \quad  	
	\widetilde{\triangle}:=  \sqrt{52 \cc^2-56 \cc \ck-104 \cc+20 \ck^2+40 \ck+65}. 
\end{gather*}
Here,  $\mathfrak{Z}_\ell:=\mathfrak{Z}_\ell(\underline{\mathfrak{G}} ,\underline{\mathfrak{H}},\underline{\mathfrak{F}})$ ($\ell=1,2$) are lengthy expressions that are analytic in all their variables and satisfy  $\mathfrak{Z}_\ell(0,0,0)=0$. Since $\underline{\mathfrak{G}}\in C^0([-1,0])$,   $\underline{\mathfrak{H}}\in C^0\bigl([-1,0], C^\infty(B_R(0)\times B_R(0))\bigr)$, and  $\underline{\mathfrak{F}} \in C^0\bigl([-1,0], C^\infty(B_R(0))\bigr)$, with $\underline{\mathfrak{H}}(\tau,0,0)=0$ and $\underline{\mathfrak{F}}(\tau,0)=0$, by applying Lemma \ref{t:gmap} and Proposition \ref{t:limG}\footnote{Recall that $\lim_{t\rightarrow t_m}\mathfrak{G}(t)=0$, and by the continuity of $g$ and $g^{-1}$, we have  $\lim_{\tau\rightarrow 0-} g^{-1}(-\tau)=g^{-1}(0)=t_m$. }, we obtain
\begin{equation}\label{e:lim1}
	\lim_{\tau \rightarrow 0-}\underline{\mathfrak{G}}(\tau)=\lim_{\tau \rightarrow 0-}\mathfrak{G}(g^{-1}(-\tau))=0,
\end{equation} 
Thus, defining $\widetilde{\mathfrak{Z}}_\ell(\tau,\uuo,\uu):=\mathfrak{Z}_\ell(\underline{\mathfrak{G}}(\tau) ,\underline{\mathfrak{H}}(\tau,\uuo,\uu),\underline{\mathfrak{F}}(\tau,\uu))$, we conclude that $\widetilde{\mathfrak{Z}}_\ell$ is continuous in $(\tau,\uuo,\uu)\in[-1,0)\times B_R(0)\times B_R(0) $ and satisfies
\begin{equation}\label{e:lim0}
	\lim_{(\tau,\uuo,\uu)\rightarrow (0,0,0)}\widetilde{\mathfrak{Z}}_\ell(\tau,\uuo,\uu)=0.
\end{equation}
Since $\cb>0$,  $1<\cc < 3/2$ and $3 c-\sqrt{2} \sqrt{8 \cc-5}<\ck<3 c+\sqrt{2} \sqrt{8 \cc-5} $, we have
\begin{equation*}
	\tilde{\lambda}_{1,2}=\frac{\cb (4 c-4 \ck- 5 \pm \widetilde{\triangle} )}{2 (2 \cc-3)}> 0 \AND
	\tilde{\lambda}_3=\frac{2\cb }{3-2\cc}> 0 .
\end{equation*}
Therefore, by   \eqref{e:lim1} and \eqref{e:lim0},  if we take $\varepsilon$ satisfying
\begin{equation*}
	0<\varepsilon< \min \bigl\{\tilde{\lambda}_m \;|\; m=1,2,3\bigr\},
\end{equation*}
then there exist small constants $\tau_\delta\in(-1,0)$ and $\tilde{R}\in(0,R)$ such that for $\tau \in (\tau_\delta,0)$ and $|(\uuo,\uu)|<\tilde{R}$, we have $|\mathfrak{Z}_\ell|<\varepsilon$ and $ |\underline{\mathfrak{G}}/B |<\varepsilon$.  Furthermore, let 
\begin{equation*}
	\Gamma:=\{(\tau,\uuo,\uu)\;|\;\tau \in (\tau_\delta,0) ,\; |(\uuo,\uu)|<\tilde{R}\},
\end{equation*}
Then, if $(\tau,\uuo,\uu)\in \Gamma$, we obtain $0<\tilde{\lambda}_m-\varepsilon<\lambda_m<\tilde{\lambda}_m+\varepsilon $ ($m=1,2,3$). In this case, if we take
\begin{equation*}
	0<\kappa \leq  \frac{1}{A}\min_{  m=1,2,3}\{\tilde{\lambda}_m - \varepsilon \} \AND  \gamma_2 \geq  \frac{1}{\kappa A}\max_{  m=1,2,3}\{\tilde{\lambda}_m +\varepsilon\} ,
\end{equation*}
then we obtain
\begin{align*}
(\Omega\zeta)^T \mathds{1} \Omega\zeta \leq \frac{1}{\kappa} (\Omega\zeta)^T  \Lambda \Omega  \zeta \leq\gamma_{2}(\Omega\zeta)^T\mathds{1} \Omega\zeta, \quad \text{i.e.,}\quad \zeta^T \mathds{1} \zeta \leq \frac{1}{\kappa} \zeta^T\mathfrak{B}\zeta \leq\gamma_{2}\zeta^T\mathds{1} \zeta
\end{align*}
for any $\tau \in (\tau_\delta,0)$, $|(\uuo,\uu)|<\tilde{R}$, and $\zeta\in \Rbb^{n+2}$. Therefore, we have verified that the system \eqref{e:fuc} is a Fuchsian system given by \eqref{e:model1} in Appendix \ref{s:fuc} for $(\tau,x^i) \in (\tau_\delta,0) \times \Tbb^n$ and $\U \in B_{\tilde{R}}(\Rbb^N)$.

\subsection{Step $5$: The  existence of the solution near $\tau=0$.}
After verifying that \eqref{e:fuc} is a Fuchsian system for $(\tau,x^i) \in (\tau_\delta,0) \times \Tbb^n$ and $\U \in B_{\tilde{R}}(\Rbb^N)$, we can directly apply Theorem \ref{t:fuc}. Specifically, there exist small constants $\sigma, \sigma_0\in(0,\tilde{R}/\kappa_s)$ with $\sigma<\sigma_0$ such that, if 
\begin{equation*}
	\|\U(\tau_\delta)\|_{H^s} \leq \sigma  \quad (\text{i.e. }\|\U(\tau_\delta)\|_{\Li} \leq \kappa_s \|\U(\tau_\delta)\|_{H^s} \leq \kappa_s \sigma<\tilde{R}) ,
\end{equation*}
then there exists a unique solution
\begin{equation}\label{e:sltdel}
	\U\in C^0([\tau_\delta,0),H^s(\Tbb^n) ) \cap C^1([\tau_\delta,0),H^{s-1}(\Tbb^n) ) \cap \Li ([\tau_\delta,0),H^k(\Tbb^n))
\end{equation}
of the initial value problem \eqref{e:model1}--\eqref{e:model2}. Moreover, for $\tau_\delta \leq \tau <0$, the solution $\U$ obeys the energy estimate
\begin{equation}\label{e:ineq1ab}
	\|\U(\tau)\|_{H^s}  \leq C(\sigma_0,\sigma_0^{-1}) \|\U(\tau_\delta)\|_{H^s} < C_1 \sigma
\end{equation}
where we take $C_1>\max\{C(\sigma_0,\sigma_0^{-1}),
1\}$. By shrinking $\sigma>0$, we are able to make sure $C_1 \sigma<1$ and $\|\U(\tau)\|_{\Li} \leq \kappa_s C_1 \sigma<\tilde{R}$.

\subsection{Step $6$:  The existence of the solution for $\tau \in[-1,0)$.}

Note that when $\tau\in[-1,\tau_\delta)$, the system \eqref{e:fuc} may not be a Fuchsian system because, as shown in Step $4$, Condition \eqref{c:5} may be violated. Consequently, we cannot use Theorem \ref{t:fuc} to prove the existence of $\U$ for $\tau\in[-1,0)$. To address this issue, we now prove the following lemma to establish the global solution for $\tau\in[-1,0)$.

\begin{lemma}\label{t:whU}
	Let $\sigma>0$ be a constant, and let $C_1$ be the constant given in Step 5. Let $\kappa_s$ denote the Sobolev constant from Theorem \ref{t:locext}. Then, there exists a small constant $\sigma_\star \in(0,1)$ such that, if the initial data satisfies $\|\U(-1)\|_{H^s} \leq (1/2) \sigma_\star\sigma$, then
	there exists a classical solution $\U\in C^1 ([-1,0) \times \Tbb^n)$ to the system \eqref{e:fuc},
	and 
	\begin{equation*}
		\U\in C^0([-1,0), H^s) \cap C^1([-1,0), H^{s-1}).
	\end{equation*}
Moreover, there exists a constant $C>1$ such that, for $\tau\in[-1,0)$,
	\begin{equation}\label{e:uest2}
	\|\U(\tau)\|_{W^{1,\infty}} \leq 2\kappa_s
	\|\U(\tau)\|_{H^s}  \leq C  \sigma.
	\end{equation}
\end{lemma}

\begin{proof}
Firstly, according to the assumption of this lemma, we note that the initial data satisfies  $\|\U(-1)\|_{H^s} < \sigma /2 $. Then, by applying Theorem \ref{t:locext} (with an appropriate shift of the time coordinate), we obtain that for the constant $R_1:=\max\{\sigma,2\kappa_s \sigma\}>\kappa_s  \sigma/2 > \kappa_s\|\U(-1)\|_{H^s}$, there exists a  $\tau^\star \in(-1,0]$, which we assume to be the maximal time, such that there exists a classical solution  $\U\in C^1([-1,\tau^\star) \times \Tbb^n)$ to the system \eqref{e:fuc} satisfying the bound
\begin{equation*}\label{e:supu}
	\sup_{(\tau,x^i)\in[-1,\tau^\star)\times \Tbb^n} |\U(\tau,x^i)| \leq R_1,
\end{equation*}
and
\begin{equation*}
	\U\in C^0([-1,\tau^\star), H^s) \cap C^1([-1,\tau^\star), H^{s-1}).
\end{equation*}

We claim that either  $(1)$  $\|\U(\tau)\|_{H^s}<\sigma$ for all $\tau\in[-1,\tau^\star)$ or $(2)$ there exists a first time $\tau_\star\in[-1,\tau^\star)$ such that $\|\U(\tau_\star)\|_{H^s}=\sigma$. If the first case holds, we set  $\tau_\star=\tau^\star$. , then, by Sobolev’s inequality ($\|\U(\tau)\|_{W^{1,\infty}}\leq 2\kappa_s\|\U(\tau)\|_{H^s}$), we obtain that for $\tau\in[-1,\tau_\star)$,  
\begin{equation}\label{e:bdu}
	\max\{\|\U(\tau)\|_{W^{1,\infty}},\|\U(\tau)\|_{H^s}\}< R_1 .
\end{equation}

For the first case $(1)$, we claim that $\tau^\star=0$. Otherwise, suppose that  $\tau^\star<0$. Since $\|\U(\tau)\|_{W^{1,\infty}} \leq 2\kappa_s
\|\U(\tau)\|_{H^s} < R_1<\infty$ for all $\tau\in[-1,\tau^\star)$, it follows from the continuation principle (see \cite[Theorem $2.2$]{Majda2012}) that $\U$ can be extended beyond $\tau^\star$ to a larger interval $[-1,\tau^\prime)$ with $\tau^\prime>\tau^\star$. This contradicts the maximality of $\tau^\star$. Thus, we conclude that $\tau^\star=0$. In this case, the lemma is proved.

Let us focus on the second case $(2)$ for the rest of the proof.
Since \eqref{e:fuc} is a special symmetric hyperbolic system satisfying conditions given by Theorem \ref{t:locest} in Appendix \ref{s:fuc}, we can apply this theorem to conclude that the solution  $\U\in C^1_b([-1,\tau_\star] \times \Tbb^n)$ to \eqref{e:fuc} satisfies the energy estimate
\begin{equation}\label{e:uest1}
	\|\U(\tau)\|_{H^s} \leq   \|\U(-1)\|_{H^s} (-\tau)^{-c_1} e^{c_2 (\tau+1)}
\end{equation}
for $\tau\in [-1,\tau_\star]$, where $c_1:=c_1(\|\U \|_{\Li([-1,\tau_\star],W^{1,\infty})})$ and  $c_2:=c_2(\|\U \|_{\Li([-1,\tau_\star],L^\infty)})$  are two finite positive constants  due to the bound \eqref{e:bdu}. Take $\sigma^\star$ satisfying\footnote{Note that $\tau_\delta\in(-1,0)$ is close to $0$. Hence $\sigma_\star$ is also small. }
\begin{equation*}
	0<\sigma_\star\leq e^{-c_2(R_1)(1+\tau_\delta)}(-\tau_\delta)^{c_1(R_1)}
\end{equation*}
where $\tau_\delta$ is defined in Step $4$ and $5$.
Then if we further let the initial data satisfy $\|\U(-1)\|_{H^s} \leq (1/2) \sigma_\star\sigma$, we  \textit{claim}: 
\begin{equation*}
	\tau_\star  > \tau_\delta. 
\end{equation*}
We prove this by contradiction. Suppose that  $\tau_\star  \leq \tau_\delta$, then for $-1<\tau<\tau_\star  \leq \tau_\delta<0$, by the Sobolev embedding theorem, the estimate  \eqref{e:uest1} for $\tau\in[-1,\tau_\star ]$ and noting that  $c_\ell=c_\ell(\|\U \|_{\Li([-1,\tau_\star],W^{1,\infty})})\leq c_\ell(R_1)$ ($\ell=1,2$), i.e., $(-\tau_\delta)^{c_1(R_1)} \leq (-\tau_\delta)^{c_1}\leq (-\tau)^{c_1}$ and $e^{-c_2(R_1)(1+\tau_\delta)}\leq e^{-c_2 (1+\tau_\delta)}\leq e^{-c_2 (1+\tau)}$, we deduce
\begin{equation*}
	 \|\U(\tau)\|_{H^s} \leq   \frac{1}{2} \sigma_\star\sigma (-\tau)^{-c_1} e^{c_2 (1+\tau)}  \leq   \frac{1}{2}  e^{-c_2(R_1)(1+\tau_\delta)}(-\tau_\delta)^{c_1(R_1)} \sigma  (-\tau)^{-c_1} e^{c_2 (1+\tau)} \leq  \frac{1}{2}\sigma
\end{equation*}
for $\tau\in [-1,\tau_\star  ]$. Then, we get $ \|\U(\tau_\star)\|_{H^s} \leq \sigma/2\neq  \sigma $.
This result contradicts   the definition of $\tau_\star$. Hence,  $\tau_\star > \tau_\delta$. From this, we conclude that $\|\U(\tau_\delta)\|_{H^s} <\sigma$ by the definition of $\tau_\star$.  Combining this with the existence interval  from  Step $5$, where there is a solution $\U$ for $\tau\in[\tau_\delta,0)$ satisfying \eqref{e:sltdel} and \eqref{e:ineq1ab}, we conclude that there exists a classical solution $\U\in C^1_b([-1,0) \times \Tbb^n)$ to the system \eqref{e:fuc}
and
\begin{equation*}
	\U\in C^0([-1,0), H^s) \cap C^1([-1,0), H^{s-1}).
\end{equation*}
Moreover, there is a constant $C>1$, such that for $\tau\in[-1,0)$,
\begin{equation*}
	\|\U(\tau)\|_{W^{1,\infty}} \leq 2\kappa_s \|\U(\tau)\|_{H^s}  \leq C  \sigma.
\end{equation*}
We have completed the proof.
\end{proof}

\subsection{Step $7$: Proof of Theorem \ref{t:mainthm1}}
\begin{proof}[Proof of Theorem \ref{t:mainthm1}]
	Using \eqref{e:ww}--\eqref{e:wi} and \eqref{e:u}--\eqref{e:ui}, we derive the following transformations
	\begin{gather}
		\varrho(t,x^i)=f(t)+f(t)u(t,x^i), \quad 	\del{t}\varrho(t,x^i)=f_0(t)+f_0(t)u_0(t,x^i) \label{e:trf1} \intertext{and} \partial_{i}\varrho(t,x^i)= \frac{1+f(t)}{\cm} u_i(t,x^i). \label{e:trf2}
	\end{gather}

	Taking $\sigma,\; \sigma_\star>0$ as defined above, if the data satisfies \eqref{e:mtdata}, then   using the transformations \eqref{e:trf1}--\eqref{e:trf2}, along with \eqref{e:udvar} and  $\U=\p{\uuo,\uui,\uu}^T$, \eqref{e:mtdata} implies
	\begin{equation*}
		\|(u_0,u_i,u)|_{t=t_0}\|_{H^s}=\Bigl\|\frac{\mathring{\varrho}}{\mf}-1\Bigr\|_{H^s(\Tbb^n)}+  	\Bigl\|\frac{\mathring{\varrho}_0}{\mf_0}-1\Bigr\|_{H^s(\Tbb^n)}+  	\Bigl\|\frac{\cm  \mathring{\varrho}_i}{1+\mf}\Bigr\|_{H^s(\Tbb^n)} \leq \frac{1}{2} \sigma_\star\sigma,
	\end{equation*}
which means that \eqref{e:mtdata} is equivalent to  $\|\U(-1)\|_{H^s} \leq (1/2) \sigma_\star\sigma$.
	By Lemma \ref{t:whU} in Step $6$, and using \eqref{e:uest2}, \eqref{e:udvar} and  $\U=\p{\uuo,\uui,\uu}^T$,  we obtain for $t\in[t_0,t_m)$,
	\begin{equation}\label{e:uest3}
		\|(u_0,u_i,u)\|_{W^{1,\infty}} \leq 2\kappa_s\|(u_0,u_i,u)\|_{H^s}  \leq C  \sigma.
	\end{equation}

Then, by using the transformations \eqref{e:trf1}--\eqref{e:trf2}, under this initial data, the estimate \eqref{e:uest3} becomes: 
\begin{align*}
 	&\Bigl\|\frac{\varrho}{f}-1\Bigr\|_{W^{1,\infty}(\Tbb^n)}+  	\Bigl\|\frac{\del{t}\varrho}{f_0}-1\Bigr\|_{W^{1,\infty}(\Tbb^n)}+ 	\Bigl\|\frac{\cm \del{i}\varrho}{1+f}\Bigr\|_{W^{1,\infty}(\Tbb^n)} \notag  \\
 	& \hspace{1cm} \leq 2\kappa_s\biggl( \Bigl\|\frac{\varrho}{f}-1\Bigr\|_{H^s(\Tbb^n)}+  	\Bigl\|\frac{\del{t}\varrho}{f_0}-1\Bigr\|_{H^s(\Tbb^n)}+ 	\Bigl\|\frac{\cm \del{i}\varrho}{1+f}\Bigr\|_{H^s(\Tbb^n)} \biggr) \leq   C \sigma
\end{align*}
for $t\in[t_0,t_m)$.
This further implies
\begin{equation*}
	\|\varrho-f\|_{L^\infty} \leq C\sigma f, \quad \|\del{t}\varrho-f_0\|_{L^\infty} \leq C\sigma f_0 \AND \|\del{i} \varrho\|_{\Li} \leq C\sigma f.
\end{equation*}
As mentioned earlier, by shrinking $\sigma$, we ensure that $C\sigma < 1$, so we have
\begin{equation*}
	0<(1- C \sigma) f \leq \varrho\leq (1+ C \sigma) f  \AND	(1- C \sigma) f_0 \leq \del{t}\varrho \leq (1+ C \sigma) f_0 .
\end{equation*}
Taking the limit as $t \to t_m$ and using \eqref{e:limf} from Theorem \ref{t:mainthm0}, we obtain
\begin{equation*}
	\lim_{t\rightarrow t_m} \varrho(t,x^i) =+\infty \AND \lim_{t\rightarrow t_m} \del{t}\varrho (t,x^i)  =+\infty.
\end{equation*}
Thus, we complete the proof of this theorem. 
\end{proof}


\appendix

\section{Preliminaries on ordinary differential equations} \label{s:AppODE}

In this section, we review some fundamental theorems on the existence, uniqueness, and continuation of solutions to ordinary differential equations (ODEs), without providing proofs, which can be found in various references on ODEs (e.g., see \cite{Hoermander1997, Hsu2013}).   Let $f$ be a continuous function defined in a neighborhood of $(t_0,y_0) \subset \Rbb\times \Rbb^n$, with values in $\Rbb^n$. We focus on
the following initial value problem
\begin{align}\label{e:ode}
	\frac{dy(t)}{dt}= f(t,y(t)); \quad
	y(t_0)= y_0 
\end{align}
for $t$ in a neighborhood of $t_0$. For given constants $a>0$, $b>0$ and $M>0$, let
\begin{align*}
	D:=\{(t,y)\in \Rbb\times \Rbb^n\;|\; |t-t_0|\leq a, |y-y_0|\leq b\}.
\end{align*}
We assume that  $f$ is defined on $D$ and that  there exists a constant $M>0$ such that
\begin{align}\label{e:bdcond}
	|f(t,y)|\leq M, \quad \text{for }  (t,y) \in D.
\end{align}

\begin{theorem}[Existences and uniqueness of ODEs]\label{t:exunqthm}
Assume \eqref{e:bdcond} and the Lipschitz condition
\begin{equation*}
	|f(t,y)-f(t,z)| \leq C|y-z|, \quad \text{if} \quad |t-t_0| \leq a, |y-y_0| \leq b, |z-y_0| \leq b.
\end{equation*}
Then, there exists a unique $C^1$ solution to the initial value problem \eqref{e:ode} for $|t-t_0|\leq T$ if $T\leq \min\{a,b/M\}$.
\end{theorem}

\begin{theorem}[Continuation of solutions]\label{t:contthm1}
Let $f\in C(D)$ and satisfy \eqref{e:bdcond}. Suppose $\phi$ is a solution of \eqref{e:ode} on the interval $J=(a,b)$. Then

$(1)$ $\lim_{t\rightarrow a+} \phi(t)=\phi(a)$ and $\lim_{t\rightarrow b-} \phi(t)=\phi(b)$ both exist and finite.

$(2)$ if $(b,\phi(b)) \in D$, then the solution $\phi$ can be continued to the right passing through the point $t=b$.
\end{theorem}

\begin{corollary}[Continuation principle]\label{t:contthm2}
	Let $f\in C(D)$. Suppose $\phi$ is a solution of \eqref{e:ode} on the interval $J=(a,b)$, and if there is a finite constant $M>0$, such that for every $t\in (a,b)$,
	\begin{equation*}
		|f(t,\phi(t))|\leq M<+\infty,
	\end{equation*}
then the solution $\phi$ can be continued to the right passing through the point $t=b$.
\end{corollary}

\section{Tools of analysis}\label{s:cont}
\subsection{Calculus}
The following lemma and proposition generalize the well known result ``continuous functions on a compact set are bounded ''.
\begin{lemma}\label{t:cont}
	If $y=f(x)$ is continuous for $x\in[a,\infty)$ and $\lim_{x\rightarrow \infty} f(x) = A$ where $A<\infty$ is a finite constant, then $y=f(x)$ is bounded on $[a,\infty)$.
\end{lemma}

\begin{proof}
	The proof of Lemma \ref{t:cont} is straightforward, for example, by using the definition of continuity or by introducing a new variable $x=1/y-(1-a)$ for $y\in(0,1]$ and defining  $f(x(y))|_{y=0}:=A$ to transform this statement into one involving continuous functions on a compact set. We omit the detailed proof. 
\end{proof}

By combining Lemma \ref{t:cont} with the result that ``continuous functions on a compact set are bounded'', we obtain the following proposition.

\begin{proposition}\label{t:cont2}
	If $y=f(x)$ is continuous for $x\in[a,b)$, where the constant $b>a$ is finite or infinite, and $\lim_{x\rightarrow b-} f(x) = A$, where $A<\infty$ is a finite constant, then $y=f(x)$ is bounded on $[a,b)$.
\end{proposition}

\subsection{The local existence for quasilinear symmetric hyperbolic systems}\label{s:locext}
This appendix contributes to the local existence theorem for quasilinear symmetric hyperbolic systems, without providing proofs. 
Detailed proofs can be found in, for example, \cite{Majda2012,Taylor2010,Racke2015}. The following theorem is given by \cite[\S$5$]{Racke2015}.

Let us consider the initial value problem,
\begin{gather}
	A^0(t,x,u)\del{t} u +A^j(t,x,u) \del{j} u+B(t,x,u) u= 0 ,  \label{e:hyp1}\\
	u|_{t=0}= u_0 \label{e:hyp2}
\end{gather}
where $u=(u_1,\cdots,u_N)\in \mathbb{C}^N$, $u=u(t,x)$, $t\in\Rbb$, $x\in \mathfrak{M}$ (where $\mathfrak{M}=\Rbb^n$ or $\Tbb^n$), and $A^\mu$, $B$ are complex $N\times N$-matrices and $C^\infty$-functions of their arguments $v\in \mathbb{C}^N$. The matrix $A^\mu $ is Hermitian, and $A^0 $ is positive definite, uniformly in each compact set with respect to $u$.

\begin{theorem}[Local existence theorem] \label{t:locext}
	Suppose $u_0\in H^s$, where $s\in \mathbb{Z}_{>n/2+1}$. Let $q_1:=\kappa_s \|u_0\|_{H^s}$, where $\kappa_s$ denotes the Sobolev constant\footnote{The Sobolev constant $\kappa_s$ characterizes the continuous embedding of $H^s$ into the space of uniformly bounded, continuous functions when $s>n/2$, i.e., $|u(x)|\leq \kappa_s\|u\|_{H^s}$ for $u\in H^s$. see details in \cite[\S$5$]{Racke2015}. } and $q_2>q_1$ is arbitrary but fixed. Then there exists a  $T>0$ such that there exists a classical solution $u\in C^1_b([0,T] \times \mathfrak{M})$ to the initial value problem \eqref{e:hyp1}--\eqref{e:hyp2} with the bound
	\begin{equation*}
		\sup_{(t,x)\in[0,T]\times \mathfrak{M}} |u(t,x)| \leq q_2,
	\end{equation*}
and
\begin{equation*}
	u\in C^0([0,T], H^s) \cap C^1([0,T], H^{s-1}).
\end{equation*}
where $T$ is a function of $\|u_0\|_{H^s}$ and $q_2$.
\end{theorem}


\section{Cauchy problems for Fuchsian systems}\label{s:fuc}
In this Appendix, we introduce the main tool for the analysis in this article, which is a variation of the theorem originally established in \cite[Appendix B]{Oliynyk2016a} and significantly developed in \cite{Beyer2020}. The proof is omitted, but readers can find the detailed proofs in \cite{Beyer2020} for a much more general case beyond the original theorem. Other generalizations and applications can be found in, for example,   \cite{Ames2022,Ames2022a,Beyer2021,Beyer2025,Fajman2021,Fajman2023,Fajman2025,Beyer2024,Beyer2024b,Beyer2020a,Oliynyk2021,Oliynyk2024,Liu2018b,Liu2018,Liu2022a,Liu2022,Liu2022b,Liu2018a,Liu2024a,Fournodavlos2024,Marshall2023}.

Consider the following symmetric hyperbolic system 
\begin{align}
	B^{\mu}(t,x,u)\partial_{\mu}u =&\frac{1}{t}\textbf{B}(t,x,u)\textbf{P}u+H(t,x,u)\quad&&\text{in}\;[T_{0},T_{1})\times\mathbb{T}^{n},  \label{e:model1}\\
	u =&u_{0} &&\text{in}\;\{T_{0}\}\times\mathbb{T}^{n},\label{e:model2}
\end{align}
where $T_{0}<T_{1}\leq0$, and we require the following \textbf{Conditions}\footnote{The notations in this appendix, such as  $\mathrm{O}(\cdot)$, $\mathcal{O}(\cdot)$ and $B_R(\Rbb^N)$, are defined in \cite[\S$2.4$]{Beyer2020}. }: 
\begin{enumerate}[(I)]
	\item \label{c:2} $\textbf{P}$ is a constant, symmetric projection operator, i.e., $\textbf{P}^{2}=\textbf{P}$, $\textbf{P}^{T}=\textbf{P}$ and $\partial_\mu \textbf{P}=0$.
	
	\item \label{c:3} $u=u(t,x)$ and $H(t,x,u)$ are $\mathbb R^{N}$-valued maps, $H\in C^{0}([T_{0},0],C^{\infty}(\Tbb^n\times B_R(\mathbb R^{N}),\Rbb^N))$ and it satisfies $H(t,x,0)=0$.
	
	\item \label{c:4} $B^{\mu}=B^{\mu}(t,x,u)$ and $\textbf{B}=\textbf{B}(t,x,u)$ are $\mathbb M_{N\times N}$-valued maps, and there is a constant $R>0$, such that $B^{i}\in   C^{0}([T_{0},0),C^{\infty}(\Tbb^n\times  B_R(\mathbb R^{N}), \mathbb M_{N\times N})$,  $\textbf{B}\in   C^{0}([T_{0},0],C^{\infty}(\Tbb^n\times  B_R(\mathbb R^{N}), \mathbb M_{N\times N})$, $B^{0}\in C^{1}([T_{0},0],C^{\infty}(\Tbb^n\times  B_R(\mathbb R^{N}), \mathbb M_{N\times N})$, and they satisfy
	\begin{equation*}\label{e:comBP}
		(B^{\mu})^{T}=B^{\mu},\quad [\textbf{P}, \textbf{B}]=\textbf{PB}-\textbf{BP}=0.
	\end{equation*}
	and $	B^i$ can be expanded as
	\begin{equation*}
		B^i(t,x,u)=B_0^i(t,x,u)+\frac{1}{t} B_2^i(t,x,u)
	\end{equation*}
	where $B_0^i, B_2^i\in C^{0}([T_{0},0],C^{\infty}(\Tbb^n\times  B_R(\mathbb R^{N}), \mathbb M_{N\times N}))$.
	
	Suppose there exist $\tilde{B}^0, \tilde{\mathbf{B}} \in C^0([T_0,0], C^\infty(\Tbb^n, \mathbb M_{N\times N}))$, such that
	\begin{equation*}
		[\mathbf{P}, \tilde{\mathbf{B}}] =0, \quad
		B^0(t,x,u)-\tilde{B}^0(t,x)= \mathrm{O}(u) , \quad
		\mathbf{B}(t,x,u)-\tilde{\mathbf{B}} (t,x)= \mathrm{O}(u)
	\end{equation*}
	for all $ (t,x,u) \in[T_{0},0]\times \Tbb^n\times B_R(\Rbb^N) $.
		
	Moreover, there is $\tilde{B}_2^i \in C^0([T_0,0], C^\infty(\Tbb^n, \mathbb M_{N\times N}))$, such that
	\begin{gather*}
		\mathbf{P} B_2^i(t,x,u) \mathbf{P}^\perp =   \mathrm{O}(\mathbf{P} u), \quad
		\mathbf{P}^\perp B_2^i(t,x,u) \mathbf{P} =   \mathrm{O}(\mathbf{P} u),  \\
		\mathbf{P}^\perp B_2^i(t,x,u) \mathbf{P}^\perp =   \mathrm{O}(\mathbf{P} u\otimes \mathbf{P} u),\quad
		\mathbf{P}  (B_2^i(t,x,u)-\tilde{B}_2^i(t,x)) \mathbf{P}  =   \mathrm{O}( u),
	\end{gather*}
	for all $ (t,x,u) \in[T_{0},0]\times \Tbb^n\times B_R(\Rbb^N) $, 	where
	$
	\textbf{P}^{\bot}=\mathds{1} - \textbf{P}
	$
	is the complementary projection operator.
	
	\item \label{c:5}
	There exist constants $\kappa,\,\gamma_{1},\,\gamma_{2}$ such that
	\begin{equation*}
		\frac{1}{\gamma_{1}}\mathds{1}\leq B^{0}\leq \frac{1}{\kappa} \textbf{B} \leq\gamma_{2}\mathds{1} \label{e:Bineq}
	\end{equation*}
	for all $ (t,x,u) \in[T_{0},0]\times \Tbb^n\times B_R(\Rbb^N) $.
	
	\item \label{c:6} For all $(t,x,u)\in[T_{0},0]\times\Tbb^n \times B_R(\mathbb R^{N})$, assume
	\begin{equation*}
		\textbf{P}^{\bot}B^{0}(t,\textbf{P}^\perp u)\textbf{P}=\textbf{P}B^{0}(t,\textbf{P}^\perp u)\textbf{P}^{\bot}=0.
	\end{equation*}
	\item \label{c:7}
	For each $ (t,x,u) \in[T_{0},0)\times \Tbb^n\times B_R(\Rbb^N) $, there exists an $s_0>0$ such that
	\begin{equation*}
		\Theta(s)=B^0\Bigl(t,x, u + s[B^0(t,x,u)]^{-1}\Bigl(\frac{1}{t}\mathbf{B}(t,x,u)\mathbf{P} v+H(t,x,u)\Bigr)\Bigr), \quad |s|<s_0,
	\end{equation*}
	defines a smooth curve in\footnote{$L(\Rbb^N)$  denotes the set of all linear maps from $\Rbb^N$ to itself. } $L(\Rbb^N)$. There exist
	constants $\theta$ and $\beta_{2\ell+1}\geq 0$, $\ell=0,\cdots,3$, such that the derivative $\Theta^\prime(0)$ satisfies\footnote{This condition is a reformulation of \cite[\S 3.1.\textrm{v}]{Beyer2020}. It is straightforward to check that it implies the condition \cite[\S 3.1.\textrm{v}]{Beyer2020} for the Fuchsian equation \eqref{e:model1}--\eqref{e:model2} that we are considering here.  }
	\begin{gather*}
		( u , \mathbf{P}\Theta^\prime(0)\mathbf{P} u )=\mathcal{O}\bigl(\theta u \otimes u   +|t|^{-1}\beta_1\mathbf{P}u \otimes \mathbf{P}u \bigr),  \\
		( u , \mathbf{P}\Theta^\prime(0)\mathbf{P}^\perp  u )=\mathcal{O}\biggl(\theta u \otimes u  +\frac{|t|^{-1}\beta_3}{R}\mathbf{P}u \otimes \mathbf{P}u \biggr),  \\
		( u , \mathbf{P}^\perp \Theta^\prime(0)\mathbf{P}u ) =\mathcal{O}\biggl(\theta u \otimes u   +\frac{|t|^{-1}\beta_5}{R}\mathbf{P}u \otimes \mathbf{P}u \biggr)
		\intertext{and}
	( u , \mathbf{P}^\perp  \Theta^\prime(0)\mathbf{P}^\perp  u ) =\mathcal{O}\biggl(\theta u \otimes u   +\frac{|t|^{-1}\beta_7}{R^2}\mathbf{P}u \otimes \mathbf{P}u \biggr)
	\end{gather*}	
	for all $(t,x,u )\in [T_0,0)\times \Tbb^n \times  B_R(\Rbb^N)$. 	
\end{enumerate}

Let us first give estimates on the solutions to the equations of the form \eqref{e:model1}, which, however, satisfy only some of the above conditions \eqref{c:3}--\eqref{c:4}. 
\begin{theorem}\label{t:locest}
	Suppose that $k\in \Zbb_{> \frac{n}{2}+1}$, $u_{0}\in H^{k}(\mathbb T^{n})$, and that conditions \eqref{c:3}--\eqref{c:4} are satisfied. 
	In addition, $B^{0} =\mathds{1}$, $B^i_0=0$, and $\mathbf{P}=\mathds{1}$, while $\tilde{\mathbf{B}}(t)$ and $B_2^i(t)$ depend only on $t$. 
    If there exists
	a classical solution $u\in C_b^{1}([T_{0},T_{\ast}]\times\mathbb T^{n})$ ($T_{\ast}<0$) to \eqref{e:model1}--\eqref{e:model2},
	then it satisfies the energy estimate
	\begin{equation*}
	\|u\|_{H^s} \leq 	e^{-c_2 T_0} (-T_0)^{c_1}\|u_0\|_{H^s} (-t)^{-c_1} e^{c_2 t}
	\end{equation*}
	for all $t\in[T_0,T_\ast]$, where
	$c_1:=c_1(\|u \|_{\Li([T_0,T_\ast],W^{1,\infty})})$ and  $c_2:=c_2(\|u(t)\|_{\Li([T_0,T_\ast],L^\infty)})$ are constants\footnote{This type of constants are non-negative, non-decreasing, continuous functions in all their arguments. }.
\end{theorem}
\begin{proof}
	According to the assumptions of this theorem, since $B^{0} =\mathds{1}$, $\mathbf{P}=\mathds{1}$, and $B^i_0=0$, equation \eqref{e:model1} reduces to 
\begin{equation}\label{e:model3}
	\partial_{0}u+\frac{1}{t} B_2^i(t)\partial_{i}u = \frac{1}{t}\textbf{B}(t,u) u+H(t,u) .
\end{equation}	
Applying $D^\alpha$ to \eqref{e:model3} yields 
\begin{equation}\label{e:model5}
	\partial_{0}D^\alpha u+\frac{1}{t} B_2^i(t)\partial_{i}D^\alpha u = \frac{1}{t}  \textbf{B}(t,u) D^\alpha u+ \frac{1}{t} [D^\alpha, \textbf{B}(t,u)] u+D^\alpha H(t,u) .
\end{equation}
Then applying $\sum_{0\leq |\alpha|\leq s}\la D^\alpha u,\cdot\ra = \sum_{0\leq |\alpha|\leq s} \int_{\Tbb^n} (D^\alpha u,\cdot )d^nx$ to both sides of \eqref{e:model5}, 
\begin{align*}
	&\sum_{0\leq |\alpha|\leq s}\la D^\alpha u, \partial_{0}D^\alpha u \ra +\frac{1}{t}  \sum_{0\leq |\alpha|\leq s}\la D^\alpha u,B_2^i(t)\partial_{i}D^\alpha u \ra = \frac{1}{t} \sum_{0\leq |\alpha|\leq s}\la D^\alpha u, \textbf{B}(t,u) D^\alpha u \ra \notag  \\
	&\hspace{1cm} + \frac{1}{t} \sum_{0\leq |\alpha|\leq s}\la D^\alpha u, [D^\alpha, \textbf{B}(t,u)] u \ra +\sum_{0\leq |\alpha|\leq s}\la D^\alpha u, D^\alpha H(t,u) \ra .
\end{align*}
Recalling the notations in \S\ref{s:funsp}, 
\begin{gather*}
	\sum_{0\leq |\alpha|\leq s}\la D^\alpha u, \partial_{0}D^\alpha u \ra =   \frac{1}{2} \partial_{0} \sum_{0\leq |\alpha|\leq s}\la D^\alpha u, D^\alpha u \ra= \frac{1}{2} \partial_{0} \|u\|_{H^s}^2
	\intertext{and}
	\sum_{0\leq |\alpha|\leq s}\la D^\alpha u, B^i_2(t) \partial_{i}D^\alpha u \ra = 0 .
\end{gather*}
Using H\"older's inequality and the standard Moser-type estimates for compositions (see, for example, \cite[Lemma $A.3$]{Liu2018} and \cite{Taylor2010}), we arrive at 
\begin{gather*}
	\sum_{0\leq |\alpha|\leq s}\la D^\alpha u, D^\alpha H(t,u) \ra
	\leq   \| u\|_{H^s} \|  H(t,u)\|_{H^s}
	\leq   	C( \|  H\|_{C^{s }(\overline{B_R(\Rbb^N)})}, \|u\|_{L^\infty}) \| u\|_{H^s}^2, \\
	\sum_{0\leq |\alpha|\leq s}\la D^\alpha u, \textbf{B}(t,u) D^\alpha u \ra \leq C\|\textbf{B}(t,u)\|_{L^\infty}\| u\|_{H^s}^2 ,
\end{gather*}
and from the inequality $\|[D^\alpha, f]g\|_{L^2} \leq C(\|Df\|_{L^\infty}\| g\|_{H^{s-1}}+\| D f\|_{H^{s-1}} \|g\|_{L^\infty})$,
\begin{align*}
&	\sum_{0\leq |\alpha|\leq s}\la D^\alpha u, [D^\alpha, \textbf{B}(t,u)] u \ra \leq 	\sum_{0\leq |\alpha|\leq s}\| D^\alpha u\|_{L^2} \|[D^\alpha, \textbf{B}(t,u)] u \|_{L^2}  \notag  \\
	&\hspace{1cm}	\leq C \| u\|_{H^s} (\|D\textbf{B}(t,u)\|_{L^\infty}\| u\|_{H^{s-1}}+\| D \textbf{B}(t,u)\|_{H^{s-1}} \|u\|_{L^\infty}) \notag  \\
	&\hspace{1cm}	\leq C \| u\|_{H^s} (\|D\textbf{B}(t,u)\|_{L^\infty}\| u\|_{H^{s-1}}+	C( \|D_u  \textbf{B}\|_{C^{s-1}(\overline{B_R(\Rbb^N)})}, \|u\|_{L^\infty})\| u\|_{H^{s }} \|u\|_{L^\infty})\notag  \\
	&\hspace{1cm}	\leq   C( \|D_u  \textbf{B}\|_{C^{s-1}(\overline{B_R(\Rbb^N)})}, \|u\|_{W^{1,\infty}}) \| u\|_{H^s}^2 .
\end{align*}
Then, combining the above estimates and noting that $t<T_\ast<0$, we obtain, for any $t\in[T_0,T_\ast]$, 
\begin{equation*}
	\partial_{0} \|u\|_{H^s}^2 \leq - \frac{2c_1 }{t}\|u\|_{H^s}^2 +2c_2 \|u\|_{H^s}^2.
\end{equation*}
where $c_1:=c_1(\|u \|_{\Li([T_0,T_\ast],W^{1,\infty})})$ and  $c_2:=c_2(\|u\|_{\Li([T_0,T_\ast],L^\infty)})$   are constants depending on $\|u\|_{\Li([T_0,T_\ast),W^{1,\infty})}$.
This inequality is equivalent to
\begin{equation*}
	\del{t}\bigl((-t)^{2c_1} e^{-2 c_2 t} \|u\|_{H^s}^2\bigr) \leq 0.
\end{equation*}
Integrating this inequality, we conclude that for any $t\in[T_0,T_\ast]$,
\begin{equation*}
	\|u\|_{H^s} \leq 	e^{-c_2 T_0} (-T_0)^{c_1}\|u_0\|_{H^s} (-t)^{-c_1} e^{c_2 t}  .
\end{equation*}
Thus, the proof is complete. 
\end{proof}

Now, let us present the global existence theorem for the Fuchsian system (see \cite{Beyer2020} for a detailed proof).
\begin{theorem}\label{t:fuc}
	Suppose that $k\in\Zbb_{> \frac{n}{2}+3}$, $u_{0}\in H^{k}(\mathbb T^{n})$, and that conditions \eqref{c:2}--\eqref{c:7} are satisfied. Additionally, assume that the constants $\kappa, \gamma_1,  \beta_1,\beta_3,\beta_5,\beta_7$ from  conditions \eqref{c:2}--\eqref{c:7} satisfy
	\begin{equation*}
		\kappa>\frac{1}{2} \gamma_1 \max\Bigl\{\sum^3_{\ell=0} \beta_{2\ell+1}, \beta_1+2k(k+1)  \mathtt{b} \Bigr\}
	\end{equation*}
where
\begin{equation*}
	\mathtt{b}:= \sup_{T_0\leq t<0} \bigl(\||\mathbf{P} \tilde{\mathbf{B}} D (\tilde{\mathbf{B}}^{-1} \tilde{B}^0)(\tilde{B}^0)^{-1} \mathbf{P} \tilde{B}_2^i \mathbf{P}|_{\mathrm{op}}\|_{\Li}+\||\mathbf{P}\tilde{\mathbf{B}} D (\tilde{\mathbf{B}}^{-1} \tilde{B}_2^i)\mathbf{P}|_{\mathrm{op}}\|_{\Li}\bigr).
\end{equation*}
	Then, there exist constants $\delta_0, \delta>0$ satisfying $\delta<\delta_0$, such that if
	\begin{equation*}
		\|u_0\|_{H^k} \leq \delta,
	\end{equation*}
    there exists a unique solution
	\begin{equation*}
		u\in C^0([T_0,0),H^k(\Tbb^n) ) \cap C^1([T_0,0),H^{k-1}(\Tbb^n) ) \cap \Li ([T_0,0),H^k(\Tbb^n))
	\end{equation*}
to the initial value problem \eqref{e:model1}--\eqref{e:model2} such that $\Pbp u(0):=\lim_{t\nearrow 0}\Pbp u(t)$ exists in $H^{s-1}(\Tbb^n)$.

\noindent Moreover, for $T_0\leq t<0$, the solution $u$ satisfies the energy estimate
\begin{equation*}\label{e:ineq1}
	\|u(t)\|_{H^k}^2  - \int^t_{T_0}\frac{1}{\tau}\|\mathbf{P} u(\tau)\|^2_{H^k}d\tau \leq C(\delta_0,\delta_0^{-1}) \|u_0\|^2_{H^k} .
\end{equation*}
\end{theorem}

\section*{Acknowledgement}
		C.L. is supported by the Fundamental Research Funds for the Central Universities, HUST: 2020kfyXJJS037 ($5003011036$ and $5003011047$).
		We also thank the referees for their comments and criticisms, which have served to improve the content and exposition of this article.

\bigskip

\textbf{Data Availability} Data sharing is not applicable to this article as no datasets were
generated or analysed during the current study.

\bigskip

\textbf{Declarations}

\bigskip

\textbf{Conflict of interest} The authors declare that they have no conflict of interest.
		
\bibliographystyle{amsplain}
\bibliography{Reference_Chao}

\end{document}